\documentclass[11pt]{article}
\usepackage[a4paper,margin=22mm]{geometry}
\usepackage{amsmath,amssymb,amsthm}
\usepackage{enumitem}
\usepackage{tikz}
\usetikzlibrary{arrows.meta,positioning,shapes.geometric,decorations.pathreplacing,calc,fit,backgrounds}
\usepackage[unicode,colorlinks=true,linkcolor=blue!60!black,urlcolor=blue!60!black,citecolor=blue!60!black]{hyperref}

\newtheorem{theorem}{Theorem}[section]
\newtheorem{lemma}[theorem]{Lemma}

\newtheorem{claim}{Claim}
\newtheorem{conjecture}{Conjecture}
\theoremstyle{definition}
\newtheorem{definition}[theorem]{Definition}
\theoremstyle{remark}

\def\eps{\varepsilon}

\tikzset{
  vtx/.style={circle,fill=black,inner sep=1.3pt},
  bigvtx/.style={circle,fill=black,inner sep=1.8pt},
  lowvtx/.style={circle,fill=red!70!black,inner sep=1.5pt},
  upvtx/.style={circle,fill=blue!70!black,inner sep=1.5pt},
  box/.style={draw,rounded corners=3pt,align=center,inner sep=5pt,font=\small},
  arr/.style={-{Stealth[length=2.2mm]},thick},
}

\title{A proof of the Erd\H{o}s--Gallai cycle decomposition conjecture}
\author{Jaehoon Kim\thanks{Department of Mathematical Sciences, KAIST, South Korea. E-mail: \href{mailto:jaehoon.kim@kaist.ac.kr}{\texttt{jaehoon.kim@kaist.ac.kr}}. Supported by the National Research Foundation of Korea (NRF) grant funded by the Korea government (MSIT) No.~RS-2023-00210430.}}
\date{}

\begin{document}
\maketitle

\begin{abstract}
In the 1960s, Erd\H{o}s and Gallai conjectured that the edges of every $n$-vertex graph can be decomposed into $O(n)$ cycles and edges. We prove this conjecture. Equivalently, every Eulerian graph on $n$ vertices can be decomposed into $O(n)$ cycles, which confirms Haj\'os' conjecture up to a constant factor.
\end{abstract}

\section{Introduction}\label{sec:intro}

When can the edges of a graph be decomposed into cycles? As Veblen~\cite{veblen} observed, extending Euler's theorem, this is possible if and only if all degrees of the graph are even. We call such a graph \emph{Eulerian}. It is natural to ask next how many cycles are needed. Haj\'os conjectured that every Eulerian graph on $n$ vertices can be decomposed into at most $\lfloor (n-1)/2\rfloor$ cycles (see~\cite{lovasz}). An arbitrary graph need not have a cycle decomposition, and then one also allows single edges. In the 1960s, Erd\H{o}s and Gallai conjectured that a linear number of cycles and edges always suffices (see~\cite{egp}).

\begin{conjecture}[Erd\H{o}s and Gallai]\label{conj:eg}
The edges of every $n$-vertex graph can be decomposed into $O(n)$ cycles and edges.
\end{conjecture}

Every graph can be decomposed into an Eulerian graph and a forest. Thus Conjecture~\ref{conj:eg} is equivalent to the statement that every Eulerian graph on $n$ vertices can be decomposed into $O(n)$ cycles. A linear bound is best possible up to the constant factor. Erd\H{o}s~\cite{erdos83} observed that $(3/2-o(1))n$ cycles and edges are needed for some graphs, generalizing an example of Gallai. Erd\H{o}s included the conjecture in several of his collections of open problems, for example in~\cite{erdos81}.

Related problems for paths and for coverings have also been studied. Lov\'asz~\cite{lovasz} proved that every graph on $n$ vertices can be decomposed into at most $n/2$ paths and cycles. Gallai conjectured that every connected graph on $n$ vertices can be decomposed into at most $\lceil n/2\rceil$ paths (see~\cite{lovasz}). Pyber~\cite{pyber} proved that the edges of every $n$-vertex graph can be covered by $n-1$ cycles and edges, and Fan~\cite{fan03} proved that the edges of every Eulerian graph on $n$ vertices can be covered by $\lfloor (n-1)/2\rfloor$ cycles. These are the covering versions of Conjecture~\ref{conj:eg} and of Haj\'os' conjecture, respectively.

The decomposition problem itself was solved for two large classes of graphs. Conlon, Fox and Sudakov~\cite{cfs} proved Conjecture~\ref{conj:eg} for graphs of linear minimum degree and for random graphs. For graphs of linear minimum degree, Gir\~ao, Granet, K\"uhn and Osthus~\cite{ggko} proved the asymptotically optimal bound $(3/2+o(1))n$. For random graphs with a wide range of edge probabilities, Kor\'andi, Krivelevich and Sudakov~\cite{kks} found the asymptotically optimal number of cycles and edges (see also Glock, K\"uhn and Osthus~\cite{gko}).

For general graphs, Erd\H{o}s and Gallai observed that repeatedly removing a longest cycle gives a decomposition into $O(n\log n)$ cycles and edges. This was the best known bound for almost fifty years. Conlon, Fox and Sudakov~\cite{cfs} improved it to $O(n\log\log n)$, and Buci\'c and Montgomery~\cite{bm} improved it to $O(n\log^* n)$. In this paper, we prove Conjecture~\ref{conj:eg}.

\begin{theorem}\label{thm:eg}
The edges of every $n$-vertex graph can be decomposed into $O(n)$ cycles and edges.
\end{theorem}

Theorem~\ref{thm:eg} implies Haj\'os' conjecture up to a constant factor. Indeed, consider a decomposition of an Eulerian graph $G$ into $O(n)$ cycles and edges. If we remove the cycles from $G$, then the single edges form an Eulerian graph, since removing a cycle keeps all degrees even. This graph can be decomposed into cycles, and the number of these cycles is at most the number of single edges.

\subsection{Previous bounds and our approach}

All previous general bounds come from iteration. In each round, the current graph is decomposed into $O(n)$ cycles and a leftover graph of smaller average degree, and the next round is applied to the leftover graph. In the approach of Conlon, Fox and Sudakov~\cite{cfs}, the average degree drops from $d$ to $d^{c}$ for a constant $c<1$, and this needs $O(\log\log n)$ rounds. Buci\'c and Montgomery~\cite{bm} split the graph into expanders on $O(d\log^4 d)$ vertices each, whose orders sum to at most $2n$, and they decompose each expander $J$ into $O(\lvert J\rvert)$ cycles and $O(\lvert J\rvert\log^{O(1)}\lvert J\rvert)$ edges~\cite[Theorem~3]{bm}. So the average degree drops from $d$ to a power of $\log d$, and $O(\log^* n)$ rounds suffice. The bound allows $O(n)$ cycles in each round, so the number of rounds appears in the final bound. Buci\'c and Montgomery~\cite[Section~6]{bm} describe the edges left over by their partition into expanders as the main bottleneck of this approach.

Our proof does not proceed in rounds. We follow a single procedure which removes cycles and edges one step at a time, and we bound the number of cycles and edges used in all steps by two quantities which are each at most a constant multiple of $n$. The first quantity is a potential function of the degrees. Let
$$f(d)=1-\frac{1}{\sqrt{d+1}}.$$
The potential $\Phi(G)=\sum_{v}f(d_G(v))$ is less than $\lvert G\rvert$, and $f(d)-f(d-2)\geq d^{-3/2}$ for $d\geq 2$ (see Section~\ref{sec:deduction}). For $1/d$ in place of $d^{-3/2}$, such a function $f$ would grow like $\log d$, and the potential could be of order $n\log n$.

\begin{definition}\label{def:light}
A cycle $C$ of a graph $G$ is \emph{heavy} if there is a real number $b\geq 1$ such that at least $b^{3/2}$ vertices of $C$ have degree at most $b$ in $G$. The graph $G$ is \emph{light} if it has no heavy cycle.
\end{definition}

Removing the edges of a cycle $C$ decreases each degree on $C$ by two. If at least $b^{3/2}$ vertices of $C$ have degree at most $b$, then each of them decreases the potential by at least $b^{-3/2}$, so the potential decreases by at least one. So we can remove heavy cycles of the current graph one by one, and in total we use at most $n$ of them.

The second quantity is the number of vertices. If the current graph has a vertex of small positive degree, then we remove its edges as single edges. Otherwise, if it has no heavy cycle, then we find an induced subgraph $H$ of polynomial minimum degree (Theorem~\ref{thm:structure}) and decompose its edges into $O(\lvert H\rvert)$ cycles and edges (Theorem~\ref{thm:qoh}). After the edges of $H$ are removed, at least $\lvert H\rvert/12$ vertices of $H$ have at most two remaining edges. We remove these vertices together with their remaining edges. So the number of cycles and edges used for $H$ is at most a constant times the number of removed vertices. Each vertex is removed at most once, and this gives a linear bound.

In this way, Conjecture~\ref{conj:eg} reduces to light graphs. The lightness condition is strong, since it bounds, for all $b$ at once, the number of vertices of degree at most $b$ on every cycle. We show that this forces the subgraphs described above. The two main ingredients are the following theorems.

\begin{theorem}\label{thm:structure}
There exists $\delta_0$ such that the following holds for every light graph $G$ with $\delta(G)\geq\delta_0$. There is a nonempty set $S\subseteq V(G)$ such that $H=G[S]$ has minimum degree at least $\lvert S\rvert^{1/1000}$. Moreover, at least $\lvert S\rvert/12$ vertices of $S$ have at most two neighbors outside $S$ in $G$.
\end{theorem}

We call such a set $S$, and the graph $G[S]$, \emph{near-closed}.

\begin{theorem}\label{thm:qoh}
For every $\eps>0$, there exists $C_\eps$ such that the edges of every graph $G$ on $n$ vertices with $\delta(G)\geq n^{\eps}$ can be decomposed into at most $C_\eps n$ cycles and edges.
\end{theorem}

Theorem~\ref{thm:qoh} extends the result of Conlon, Fox and Sudakov~\cite{cfs} for graphs of linear minimum degree to graphs of polynomial minimum degree. We prove it using the robust sublinear expanders and a cycle decomposition result of Buci\'c and Montgomery~\cite{bm} together with a decomposition result of Conlon, Fox and Sudakov~\cite{cfs} for sparse graphs.

Theorem~\ref{thm:structure} is the main new ingredient, and its proof is elementary. We did not try to optimize the exponent $1/1000$ or the constant $12$.

We organize the rest of the paper as follows. In Section~\ref{sec:outline}, we deduce Theorem~\ref{thm:eg} from Theorems~\ref{thm:structure} and~\ref{thm:qoh}, and we outline the proofs of these two theorems. In Section~\ref{sec:prelim}, we collect facts on light graphs and regions. In Section~\ref{sec:counting}, we prove counting lemmas for cycles through vertices of small degree. We prove Theorem~\ref{thm:structure} in Section~\ref{sec:structure} and Theorem~\ref{thm:qoh} in Section~\ref{sec:qoh}. We make concluding remarks in Section~\ref{sec:conclusion}.

\subsection{Notation}\label{sec:notation}

All graphs are finite and simple. A \emph{decomposition} of a graph $G$ into cycles and edges is a partition of $E(G)$ into the edge sets of cycles and single edges. We write $\lvert G\rvert$ and $e(G)$ for the numbers of vertices and edges of $G$, and $N_G(X)$ for the set of vertices outside $X$ which have a neighbor in $X$. We omit the subscript $G$ when the graph is clear. All logarithms are to base two, as in~\cite{bm}. We often treat large numbers as integers by omitting floors and ceilings when this does not affect the argument. In particular, we apply lemmas whose parameters are integers with the parameters rounded up.

We use Chernoff's bound in the following form (see for example~\cite[Theorem~4]{bm}). If $\xi$ is a binomial random variable with mean $\mu$ and $0<\theta<1$, then
$$\Pr[\xi\leq(1-\theta)\mu]\leq\exp(-\theta^2\mu/2).$$

\section{Proof of \texorpdfstring{Theorem~\ref{thm:eg}}{Theorem 1.1} and outline}\label{sec:outline}

\subsection{The procedure}\label{sec:deduction}

The proof builds a decomposition of $G$ step by step. To \emph{record} a cycle or an edge means to add it to this decomposition and to delete its edges from the current graph. Figure~\ref{fig:procedure} shows the procedure.

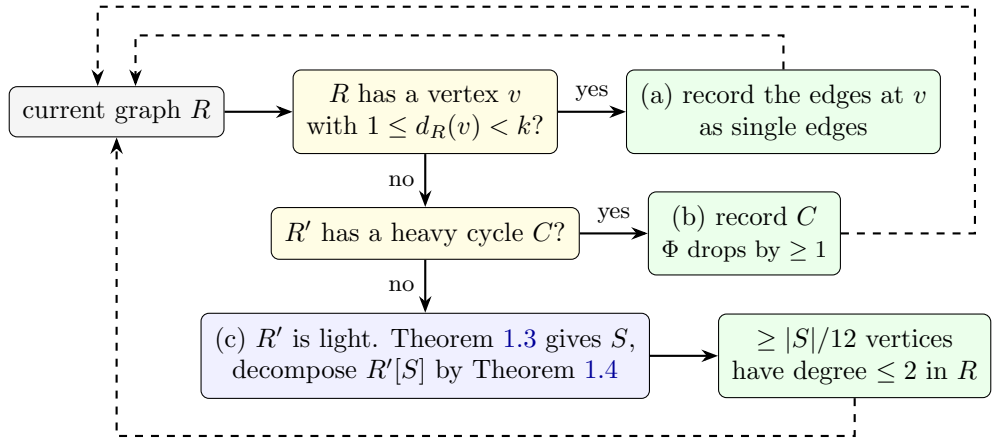
\begin{figure}[ht]
\centering
\begin{tikzpicture}[node distance=7mm and 9mm]
\node[box,fill=gray!8] (R) {current graph $R$};
\node[box,right=of R,fill=yellow!12] (q1) {$R$ has a vertex $v$\\ with $1\leq d_R(v)<k$?};
\node[box,right=of q1,fill=green!8] (a) {(a) record the edges at $v$\\ as single edges};
\node[box,below=of q1,fill=yellow!12] (q2) {$R'$ has a heavy cycle $C$?};
\node[box,right=of q2,fill=green!8] (b) {(b) record $C$\\ \footnotesize $\Phi$ drops by $\geq 1$};
\node[box,below=of q2,fill=blue!6] (c) {(c) $R'$ is light. Theorem~\ref{thm:structure} gives $S$,\\ decompose $R'[S]$ by Theorem~\ref{thm:qoh}};
\node[box,right=of c,fill=green!8] (d) {$\geq \lvert S\rvert/12$ vertices\\ have degree $\leq 2$ in $R$};
\draw[arr] (R) -- (q1);
\draw[arr] (q1) -- node[above,font=\footnotesize]{yes} (a);
\draw[arr] (q1) -- node[left,font=\footnotesize]{no} (q2);
\draw[arr] (q2) -- node[above,font=\footnotesize]{yes} (b);
\draw[arr] (q2) -- node[left,font=\footnotesize]{no} (c);
\draw[arr] (c) -- (d);
\draw[arr,dashed] (a.north) -- ++(0,0.45) -| ($(R.north)+(0.25,0)$);
\draw[arr,dashed] (b.east) -- ($(b.east -| a.east)+(0.45,0)$) |- ($(R.north)+(-0.25,1.05)$) -- ($(R.north)+(-0.25,0)$);
\draw[arr,dashed] (d.south) -- ++(0,-0.5) -| (R.south);
\end{tikzpicture}
\caption{The procedure in the proof of Theorem~\ref{thm:eg}.}
\label{fig:procedure}
\end{figure}

\begin{proof}[Proof of Theorem~\ref{thm:eg}]
Let $\delta_0$ be given by Theorem~\ref{thm:structure} and $C_{1/1000}$ by Theorem~\ref{thm:qoh} with $\eps=1/1000$, and fix an integer $k\geq\max\{\delta_0,3,C_{1/1000}\}$.
Let $G$ be a graph on $n$ vertices. We repeat the following steps on a current spanning subgraph $R$ of $G$, starting with $R=G$, until $R$ has no edges.
\begin{enumerate}[label=\upshape(\alph*)]
\item\label{step:a} If $R$ has a vertex $v$ with $1\leq d_R(v)<k$, we record every edge of $R$ incident with $v$ as a single edge.
\item\label{step:b} Otherwise, let $R'$ be the graph obtained from $R$ by removing its isolated vertices. If $R'$ has a heavy cycle $C$, we record $C$.
\item\label{step:c} Otherwise, $R'$ is a light graph, and $\delta(R')\geq k\geq\delta_0$ since step~\ref{step:a} does not apply. Theorem~\ref{thm:structure} gives a set $S\subseteq V(R')$. We record a decomposition of the edges of $R'[S]$ into at most $k\lvert S\rvert$ cycles and edges.
\end{enumerate}
In~\ref{step:c}, such a decomposition exists by Theorem~\ref{thm:qoh}, since $R'[S]$ has minimum degree at least $\lvert S\rvert^{1/1000}$ and $C_{1/1000}\leq k$. Each step removes at least one edge. For~\ref{step:c}, this holds since $S$ is nonempty and $R'[S]$ has minimum degree at least one. Thus the procedure stops, and the recorded cycles and edges form a decomposition of $G$. We now bound their number.

Each application of~\ref{step:a} turns a vertex into an isolated vertex, and isolated vertices are never used again. Thus~\ref{step:a} is applied at most $n$ times, and it records at most $kn$ edges in total.

We bound the number of applications of~\ref{step:b} by the potential $\Phi(R)=\sum_{v\in V(G)}f(d_R(v))$, where $f(d)=1-1/\sqrt{d+1}$. As $f$ is increasing, removing edges does not increase $\Phi(R)$. Initially, $\Phi(G)<n$. Suppose~\ref{step:b} removes a heavy cycle $C$ of $R'$, and let $b\geq 1$ be such that at least $b^{3/2}$ vertices of $C$ have degree at most $b$ in $R'$. Each of these vertices $v$ lies on $C$, so its degree $d=d_R(v)$ satisfies $2\leq d\leq b$, and its contribution decreases by
$$f(d)-f(d-2)=\frac{1}{\sqrt{d-1}}-\frac{1}{\sqrt{d+1}}=\frac12\int_{d-1}^{d+1}x^{-3/2}\,dx\geq d^{-3/2}\geq b^{-3/2},$$
where the first inequality holds since $x^{-3/2}$ is convex. Thus $\Phi(R)$ decreases by at least one. As $\Phi(R)\geq 0$ throughout,~\ref{step:b} is applied at most $n$ times.

Finally, consider~\ref{step:c}. By Theorem~\ref{thm:structure}, at least $\lvert S\rvert/12$ vertices of $S$ have at most two neighbors outside $S$ in $R'$. After the edges of $R'[S]$ are removed, each of these vertices has degree at most $2<k$, so~\ref{step:a} removes all its remaining edges before the next application of~\ref{step:c}, and it is isolated in all later steps and hence lies in no later set $S\subseteq V(R')$. Thus each vertex of $G$ is one of these vertices in at most one application of~\ref{step:c}. Therefore the sum of $\lvert S\rvert$ over all applications of~\ref{step:c} is at most $12n$, and~\ref{step:c} records at most $12kn$ cycles and edges.

In total, we record at most $14kn$ cycles and edges. This proves Theorem~\ref{thm:eg}.
\end{proof}

Since we test for heavy cycles again after each step, we need no statement on how lightness behaves when edges are removed.

\subsection{Outline of the proof of Theorem~\ref{thm:structure}}\label{sec:outline-structure}

We prove Theorem~\ref{thm:structure} in Section~\ref{sec:structure}. The proofs of Conlon, Fox and Sudakov~\cite{cfs} and of Buci\'c and Montgomery~\cite{bm} remove cycles which are long compared with the average degree, so they control long cycles only in terms of the average degree. Let $G$ be a light graph with large minimum degree. We use two consequences of lightness.

First, lightness rules out long cycles (Section~\ref{sec:prelim}). Let $F$ be a subgraph of $G$ in which the vertices of $G$-degree greater than $b$ form an independent set. At least half of the vertices in any cycle of $F$ have $G$-degree at most $b$, so every cycle of $F$ has length less than $2b^{3/2}$ (Lemma~\ref{lem:light-circumference}). Hence, in a depth-first search forest of $F$, each edge of $F$ joins a vertex to an ancestor at distance less than $2b^{3/2}$. For a chosen $m$, cutting this forest into pieces of about $m$ vertices thus gives disjoint regions $R_i$ of size $O(m+b^{3/2})$ which together contain almost all vertices of $F$, and small separators $S_i$ such that each $F$-neighbor of a vertex of $R_i$ lies in $R_i\cup S_i$ (Lemma~\ref{lem:regions}).

Second, only $o(n)$ vertices of small degree have two neighbors of much larger degree, provided that no vertex of large degree has many neighbors among them (Section~\ref{sec:counting}). Indeed, otherwise there are many paths $y\,x\,y'$ with $x$ of small degree and $y,y'$ of large degree, and the vertices of large degree have many other neighbors, through which we can link many of these paths into one cycle. This cycle passes through many vertices of small degree, so it is heavy (Lemmas~\ref{lem:expansion} and~\ref{lem:two-parents}). This condition cannot be dropped. If two vertices are adjacent to all vertices of many disjoint copies of $K_b$, then the graph is light for large $b$, but every vertex of these copies has two neighbors of large degree. This is why, in step~\ref{st:regions}, the vertices of large degree with many neighbors in the interior of the window are treated separately.

Using Lemma~\ref{lem:two-parents}, we show in Lemma~\ref{lem:down} that in a light graph on $n$ vertices, at most $(1/r+d^{-1/7})n$ vertices of degree at least $d$ have $r$ neighbors of much smaller degree. We use Lemma~\ref{lem:down} in steps~\ref{st:window} and~\ref{st:regions} below.

Using these two tools, we prove Theorem~\ref{thm:structure} in the five steps below, shown in Figure~\ref{fig:steps}.

\begin{figure}[ht]
\centering
\begin{tikzpicture}[x=1cm,y=1cm]
\begin{scope}
\draw[-{Stealth[length=1.6mm]}] (0,0.7)--(2.6,0.7);
\foreach \k in {0,...,5}{\draw ({0.12+0.47*\k},0.76)--({0.12+0.47*\k},0.64);}
\fill[blue!12] (0.59,0.8) rectangle (2.00,1.75);
\draw[blue!60!black] (0.59,0.8) rectangle (2.00,1.75);
\fill[red!25] (1.06,0.8) rectangle (1.53,1.3);
\draw[red!70!black] (1.06,0.8) rectangle (1.53,1.3);
\node[font=\scriptsize,blue!50!black] at (0.8,1.55) {$I$};
\node[font=\scriptsize,red!60!black] at (1.295,1.05) {$A$};
\draw[decorate,decoration={brace,amplitude=3pt}] (0.12,1.95)--(2.47,1.95) node[midway,above=2pt,font=\scriptsize]{$L$};
\node[font=\scriptsize,below] at (0.12,0.62) {$a^{1/243}$};
\node[font=\scriptsize,below] at (2.47,0.62) {$a$};
\end{scope}
\begin{scope}[xshift=3.1cm]
\draw[gray!60,rounded corners=3pt,fill=gray!6] (0,0.1) rectangle (2.6,1.6);
\foreach \a/\b in {0.08/0.80,0.96/1.64,1.80/2.52}{\draw[blue!50!black,rounded corners=2pt,fill=blue!8] (\a,0.3) rectangle (\b,1.45);}
\foreach \a in {0.82,1.66}{\fill[orange!40] (\a,0.38) rectangle ({\a+0.12},1.37);}
\draw (0.45,2.3)--(0.32,1.15); \draw (0.45,2.3)--(0.62,0.85); \draw (1.30,2.3)--(1.18,1.15); \draw (1.30,2.3)--(1.45,0.85); \draw (2.15,2.3)--(2.05,1.15); \draw (2.15,2.3)--(2.30,0.85);
\foreach \x in {0.45,1.30,2.15}{\node[upvtx] at (\x,2.3) {};}
\node[font=\scriptsize,blue!60!black,anchor=east] at (0.32,2.3) {$U$};
\node[font=\scriptsize] at (0.44,0.55) {$R_i$};
\node[font=\scriptsize] at (0.88,0.21) {$S_i$};
\node[font=\scriptsize,gray!50!black,anchor=west] at (2.6,0.25) {$F$};
\end{scope}
\begin{scope}[xshift=6.2cm]
\draw[blue!50!black,rounded corners=4pt,fill=blue!6] (0.1,0.15) rectangle (2.5,2.35);
\draw[red!70!black,dashed,rounded corners=3pt] (0.5,0.5) rectangle (2.05,1.85);
\coordinate (c1) at (0.85,1.15); \coordinate (c2) at (1.25,1.5); \coordinate (c3) at (1.7,1.2); \coordinate (c4) at (1.05,0.75); \coordinate (c5) at (1.55,0.75);
\draw (c1)--(c2)--(c3)--(c5)--(c4)--(c1); \draw (c1)--(c3); \draw (c2)--(c4); \draw (c2)--(c5); \draw (c4)--(c3);
\draw[gray] (0.3,0.35)--(c4); \draw[gray] (2.3,0.4)--(c5); \draw[gray] (0.3,2.1)--(c1); \draw[gray] (2.3,1.6)--(c3);
\foreach \c in {c1,c2,c3,c4,c5}{\node[vtx] at (\c) {};}
\foreach \q in {(0.3,0.35),(2.3,0.4),(0.3,2.1),(2.3,1.6)}{\node[circle,draw,fill=white,inner sep=1.2pt] at \q {};}
\node[font=\scriptsize,red!60!black] at (0.78,1.66) {$V_i$};
\node[font=\scriptsize,blue!50!black] at (2.2,2.12) {$R_i$};
\end{scope}
\begin{scope}[xshift=9.3cm]
\draw[red!70!black,dashed,rounded corners=3pt,fill=red!4] (0.15,0.65) rectangle (1.85,2.05);
\coordinate (x) at (1.15,1.2);
\foreach \q in {(0.55,1.4),(0.85,1.8),(1.55,1.7),(0.6,0.9)}{\draw (x)--\q; \node[vtx] at \q {};}
\draw[red!70!black,thick] (x)--(2.35,2.35); \draw[red!70!black,thick] (x)--(2.3,0.15);
\node[upvtx] at (2.35,2.35) {};
\node[lowvtx] at (2.3,0.15) {};
\node[bigvtx] at (x) {};
\node[font=\scriptsize] at (1.05,0.95) {$x$};
\node[font=\scriptsize,anchor=west] at (2.45,2.35) {$y$};
\node[font=\scriptsize,anchor=west] at (2.4,0.15) {$z$};
\node[font=\scriptsize,red!60!black] at (0.42,1.87) {$V_i$};
\end{scope}
\begin{scope}[xshift=12.4cm]
\foreach \a in {0.05,0.7,1.35,2.0}{\draw[blue!50!black,rounded corners=2pt,fill=blue!6] (\a,0.45) rectangle ({\a+0.55},1.85);}
\draw[green!50!black,line width=1.1pt,rounded corners=2pt] (0.7,0.45) rectangle (1.25,1.85);
\foreach \q in {(0.32,1.3),(0.97,1.55),(0.85,1.15),(1.1,0.95),(0.92,0.7),(1.62,1.1),(2.2,1.45),(2.4,0.8)}{\node[bigvtx] at \q {};}
\node[font=\scriptsize,green!40!black] at (0.975,2.1) {$R_i$};
\end{scope}
\foreach \k in {1,2,3,4}{\draw[arr] ({3.1*\k-0.47},1.25)--({3.1*\k-0.03},1.25);}
\foreach \k/\t in {0/{(1) window},1/{(2) regions},2/{(3) removal},3/{(4) outside neighbors},4/{(5) one region}}{\node[font=\footnotesize] at ({3.1*\k+1.3},-0.35) {\t};}
\end{tikzpicture}
\caption{The five steps in the proof of Theorem~\ref{thm:structure}. Open circles are removed vertices. In (2) and (4), degrees increase from bottom to top. In (4), a remaining vertex $x$ of $A$ has at most one neighbor $z$ of much smaller degree and at most one neighbor $y$ outside $V_i$ of much larger degree. In (5), we take $S=V_i$ for a region $R_i$ which contains many remaining vertices of $A$ (dots).}
\label{fig:steps}
\end{figure}
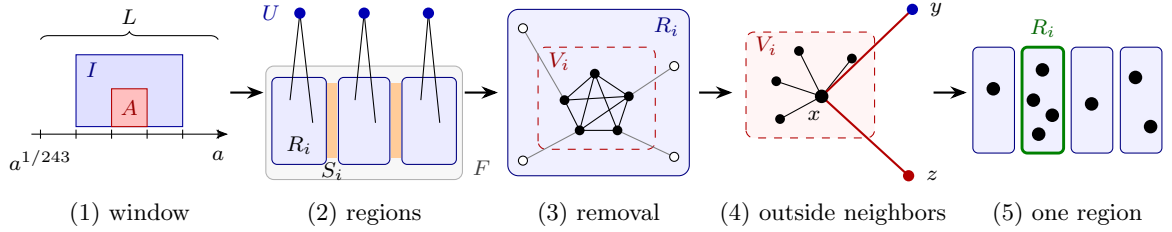

\begin{enumerate}[label=\upshape(\arabic*)]
\item\label{st:window} \emph{A degree window} (Section~\ref{sec:window}). We divide the degrees into windows $(a^{1/243},a]$, and we shift all windows together by a factor $3$ in the exponent. Averaging over the five shifts gives a window $L$ with top $a$ in which at least $\lvert L\rvert/11$ vertices lie in the middle of the window and have at most one neighbor of much smaller degree. Let $A$ be the set of these vertices. We also make sure that $\lvert L\rvert\geq n/(1000\log a)$, so that the bounds of the form $a^{-c}n$ obtained from the counting lemmas are small compared with $\lvert L\rvert$. The vertices of degree between $a^{1/81}$ and $a^{1/3}$ form the interior $I$ of the window, and $A\subseteq I$.
\item\label{st:regions} \emph{Regions} (Sections~\ref{sec:graphF} and~\ref{sec:regions}). The vertices of $I$ have degree at most $a^{1/3}$, so by Lemma~\ref{lem:down}, only a small number of vertices of degree greater than $a$ have many neighbors in $I$. Let $U$ be the set of these vertices, and let $F$ be the graph on $L\cup U$ which contains all edges of $G$ on $L\cup U$ except those inside $U$. Vertices below the window have small degree, and vertices of degree greater than $a$ outside $U$ have a small number of neighbors in $I$. So not many edges join $I$ to vertices outside $L\cup U$. The vertices of $F$ of $G$-degree greater than $a$ lie in $U$ and so form an independent set in $F$. Thus, by the first tool with $b=a$ and $m=a^2$, we find regions in $F$ with $O(a^2)$ vertices each.
\item\label{st:removal} \emph{Removing vertices of small degree} (Section~\ref{sec:peeling}). In each graph $F[R_i\cup S_i]$, we repeatedly remove vertices of small degree. The remaining set $V_i$ induces a subgraph of minimum degree at least $\tau$, where $\tau$ is a small power of $a$. The vertices of $I$ have large degree, and only a small fraction of their edges leave $F$. Since $N_F(R_i)\subseteq S_i$, all edges of $F$ at a vertex of $R_i$ lie in $F[R_i\cup S_i]$. So, by counting the edges of the removed vertices, only a small fraction of the vertices of $I$ are removed. We then discard the small number of vertices of $A$ which lie in a separator or are adjacent to a removed vertex of a separator or of $I$.
\item\label{st:outside} \emph{Neighbors outside the sets $V_i$} (Section~\ref{sec:parentcount}). Let $x$ be a remaining vertex of $A$, and let $R_i$ be the region containing $x$. At most one neighbor of $x$ has much smaller degree. Each neighbor of $x$ outside $V_i$, except this one, has degree greater than $a^{1/3}$ and fewer than $\tau$ neighbors among the remaining vertices of $A$ in all regions together. One application of the second tool to all regions shows that almost all remaining vertices of $A$ have at most one neighbor of degree greater than $a^{1/3}$ outside their set $V_i$. So most of the remaining vertices of $A$ have at most two neighbors outside their set $V_i$.
\item\label{st:choose} \emph{Choosing one region} (Section~\ref{sec:choose}). The previous step gives $\Omega(\lvert L\rvert)$ such vertices. Comparing their number with the sizes of the regions, we find a region $R_i$ which contains at least $\lvert R_i\cup S_i\rvert/12$ of these vertices, and we take $S=V_i$. The set $R_i\cup S_i$ has $O(a^2)$ vertices. By step~\ref{st:removal}, $G[S]$ has minimum degree at least $\tau$, and so at least $\lvert S\rvert^{1/1000}$, as desired.
\end{enumerate}

\subsection{Outline of the proof of Theorem~\ref{thm:qoh}}\label{sec:outline-qoh}

We prove Theorem~\ref{thm:qoh} in Section~\ref{sec:qoh}. As our proof does not proceed in rounds, it may record only $O(n)$ single edges. Let $G$ be a graph on $n$ vertices with $\delta(G)\geq n^{\eps}$. We use robust sublinear expanders in the sense of Buci\'c and Montgomery~\cite{bm} (Definition~\ref{def:expander}). An $s$-expander $J$ is a graph in which every set $X$ of at most $2\lvert J\rvert/3$ vertices has at least $\tfrac{\lvert X\rvert}{32\log^2\lvert J\rvert}$ neighbors outside $X$, even after the removal of at most $s\lvert X\rvert$ edges, where $s=n^{2\eps/3}$.

First, we use a lemma of Buci\'c and Montgomery (Theorem~\ref{thm:bm}\ref{bm:split}) to split the edges of $G$ into edge-disjoint expanders $J_1,\dots,J_{j^*}$ with $\sum_i\lvert J_i\rvert\leq 2n$ and a set $F$ of at most $4sn\log n$ other edges. The expanders may share vertices. A vertex which lies in no expander has all incident edges in $F$. As $\delta(G)\geq n^{\eps}$, at most $2\lvert F\rvert/n^{\eps}\leq 8sn\log n/n^{\eps}=n^{1-\Omega(1)}$ vertices lie in no expander. In each expander, we choose a sparse random subgraph, its \emph{reservoir}. Buci\'c and Montgomery~\cite{bm} showed that $O(n)$ cycles can be removed from any graph on $n$ vertices so that only $n\log^{O(1)}n$ edges remain (Theorem~\ref{thm:bm}\ref{bm:bulk}). We apply this to the edges not in the reservoirs, and we let $H$ be the graph of the remaining edges. The edges of $H$ between vertices which lie in no expander or have large degree in $H$ span only $O(n/\log^{600}n)$ vertices, and Theorem~\ref{thm:bm}\ref{bm:bulk} decomposes them into $o(n)$ cycles and edges.

We close almost all other edges of $H$ into cycles through the expanders (Section~\ref{sec:qoh-proof}). We assign each vertex of small degree in $H$ which lies in an expander to one expander which contains it. We decompose the edges of $H$ between vertices of small degree in $H$ which lie in expanders into paths and cut them into short pieces. A family of pieces whose ends are assigned to expanders in a cyclic order is closed into one cycle by paths inside these expanders, which we call \emph{joining paths}. We call such a family a \emph{construction}, and we find constructions as short cycles of an auxiliary multigraph on the expanders (Lemma~\ref{lem:short-cycle}). At each of the $O(n/\log^{600}n)$ vertices which have large degree in $H$ or lie in no expander, we pair the edges whose other ends are assigned to the same expander. We group these pairs (Lemma~\ref{lem:grouping}) and close them into cycles in a similar way. At most one edge remains at such a vertex for each expander, and we close these edges into cycles along the cycles of an auxiliary bipartite graph (Lemma~\ref{lem:medium-cycle}). A routing lemma (Lemma~\ref{lem:routing}) finds joining paths for all but $o(n)$ constructions at once. The joining paths are edge-disjoint, and those of one cycle have disjoint interiors which avoid the other parts of this cycle (Figure~\ref{fig:constructions}(a)). We record the edges of the other constructions as single edges.

The joining paths use edges of the reservoirs. The other edges of the expanders need no separate treatment, since the cycle removal above is applied to all edges not in the reservoirs. Conlon, Fox and Sudakov~\cite{cfs} showed that sparse graphs can be decomposed into linearly many cycles and edges (Theorem~\ref{thm:cfs}), so the edges of the reservoirs which are not used by joining paths can be decomposed at the end. The minimum degree is used in two ways. Since $\delta(G)\geq n^{\eps}$, we can take $s$ to be a power of $n$. Then a random subgraph of density $n^{-\Omega(1)}$ of an expander is again an expander (Lemma~\ref{lem:split}), so the reservoirs can be sparse enough for the result of Conlon, Fox and Sudakov. In addition, only $n^{1-\Omega(1)}$ vertices lie in no expander, as we saw above.

\section{Preliminaries}\label{sec:prelim}

In this section, we collect the facts on light graphs and regions which we use in Section~\ref{sec:structure}.

\begin{lemma}\label{lem:light-circumference}
Let $G$ be a light graph and $b\geq 1$. If $F$ is a subgraph of $G$ in which the vertices of $G$-degree greater than $b$ form an independent set, then every cycle of $F$ has length less than $2b^{3/2}$.
\end{lemma}

\begin{proof}
No two consecutive vertices of a cycle of $F$ have $G$-degree greater than $b$, so at least half of its vertices have $G$-degree at most $b$. Hence a cycle of $F$ of length at least $2b^{3/2}$ has at least $b^{3/2}$ vertices of $G$-degree at most $b$, and so it is heavy, a contradiction.
\end{proof}

The next lemma is a version of the usual separator argument for graphs of bounded treewidth.

\begin{lemma}\label{lem:regions}
Let $k\geq 3$ and $m>0$, and let $F$ be a graph in which every cycle has length less than $k$. Then there are disjoint sets $R_1,\dots,R_p$ and sets $S_1,\dots,S_p$ such that $R_i\cap S_i=\varnothing$ and $N_F(R_i)\subseteq S_i$ for each $i\in[p]$, and
$$\lvert R_i\cup S_i\rvert<2m+2k,\qquad\sum_{i\in[p]}\lvert S_i\rvert\leq\frac{2k\lvert F\rvert}{m}\qquad\text{and}\qquad\sum_{i\in[p]}\lvert R_i\rvert\geq\Big(1-\frac{k}{m}\Big)\lvert F\rvert.$$
\end{lemma}

\begin{proof}
Let $T$ be a depth-first search forest of $F$. We add a new vertex $r$ and join it to the roots of the trees of $T$, and we obtain a tree $T^+$ rooted at $r$. Ancestors and descendants refer to $T^+$. As $T$ is a depth-first search forest, one end of each edge of $F$ is an ancestor of the other end. If an edge $xy$ of $F$ joins $y$ to an ancestor $x$ at distance $j\geq 2$ in $T$, then $xy$ and the path of $T$ from $x$ to $y$ form a cycle of length $j+1$, so $j<k-1$. For a vertex $u$ of $F$, let $B_u$ be the set of $u$ and its ancestors other than $r$ at distance at most $k-1$ from $u$, and put $B_r=\varnothing$. Then $\lvert B_u\rvert\leq k$.

We first choose disjoint sets $W_1,\dots,W_p$ which cover $V(F)$, each with fewer than $2m+1$ vertices, such that the neighbors of each $W_i$ outside the earlier sets lie in a set $B_{u_i}$. Then we obtain $R_i$ from $W_i$ by removing the sets $B_{u_j}$ with $j<i$.

We choose $W_1,W_2,\dots$ one by one, and we call a vertex \emph{unassigned} if it lies in none of the sets chosen so far. If some vertex of $T^+$ has at least $m$ unassigned proper descendants, then we choose such a vertex $u_i$ farthest from $r$. Each child $u'$ of $u_i$ has fewer than $m$ unassigned proper descendants, so the subtree rooted at $u'$ has fewer than $m+1$ unassigned vertices. We go through the children of $u_i$ one by one. We put the unassigned vertices of the subtree rooted at each child into $W_i$, and we stop when $W_i$ has at least $m$ vertices. This is possible, since $u_i$ has at least $m$ unassigned proper descendants. Then $m\leq\lvert W_i\rvert<2m+1$.

Let $w\in W_i$, and let $y\notin W_i$ be a neighbor of $w$ which is unassigned before $W_i$ is chosen. Then $y$ is an ancestor of $w$, since an unassigned descendant of $w$ lies in the same subtree as $w$ and so in $W_i$. As $y\notin W_i$, it does not lie in this subtree, so $y$ is $u_i$ or an ancestor of $u_i$. As $y$ has distance less than $k-1$ from $w$ in $T$, it has distance less than $k-2$ from $u_i$, and so it lies in $B_{u_i}$.

Suppose that we stop after choosing $W_1,\dots,W_{p-1}$. Then $r$ has fewer than $m$ unassigned proper descendants, so fewer than $m$ vertices are unassigned. We let $W_p$ be the set of these vertices, and we put $u_p=r$. As $W_1,\dots,W_{p-1}$ each have at least $m$ vertices, we have $(p-1)m\leq\lvert F\rvert$. As the sets $W_1,\dots,W_p$ cover $V(F)$, the previous paragraph shows that each $i\in[p]$ satisfies
$$N_F(W_i)\subseteq B_{u_i}\cup\bigcup_{j<i}W_j.$$
In particular, a vertex of $W_i$ with a neighbor in an earlier set $W_j$ lies in $B_{u_j}$.

We put
$$R_i=W_i\setminus\bigcup_{j<i}B_{u_j}\qquad\text{and}\qquad S_i=B_{u_i}\cup\Big(W_i\cap\bigcup_{j<i}B_{u_j}\Big).$$
Then $R_i\cap S_i=\varnothing$, since $W_i$ is disjoint from $B_{u_i}$ and $R_i$ avoids the sets $B_{u_j}$ with $j<i$. Let $x\in R_i$ and let $y\notin R_i$ be a neighbor of $x$. If $y\in W_i$, then $y\in S_i$. The vertex $y$ does not lie in an earlier set $W_j$, since otherwise $x\in B_{u_j}$. If $y$ lies in a later set $W_{i'}$, then $y\in B_{u_i}\subseteq S_i$, since $y\in W_{i'}$ has a neighbor in the earlier set $W_i$. Hence $N_F(R_i)\subseteq S_i$. The sets $W_i\setminus R_i=W_i\cap\bigcup_{j<i}B_{u_j}$ are disjoint subsets of $\bigcup_{j<p}B_{u_j}$, so their sizes sum to at most $(p-1)k\leq k\lvert F\rvert/m$. The sets $B_{u_1},\dots,B_{u_{p-1}}$ also have total size at most $(p-1)k\leq k\lvert F\rvert/m$. Hence $\sum_i\lvert S_i\rvert\leq 2k\lvert F\rvert/m$. As the sets $W_i$ cover $V(F)$, we have $\sum_i\lvert R_i\rvert\geq\lvert F\rvert-k\lvert F\rvert/m=(1-k/m)\lvert F\rvert$. Finally, the sets $R_i$ are disjoint, and $R_i\cup S_i=W_i\cup B_{u_i}$, so $\lvert R_i\cup S_i\rvert<2m+1+k\leq 2m+2k$.
\end{proof}

We call the sets $R_i$ \emph{regions} and the sets $S_i$ their \emph{separators}. The sets $S_i$ need not be disjoint from each other or from the regions $R_j$ with $j\neq i$.

\section{Counting cycles through vertices of small degree}\label{sec:counting}

This section proves Lemmas~\ref{lem:two-parents} and~\ref{lem:down}, the two counting lemmas used in Section~\ref{sec:structure}. Lemma~\ref{lem:expansion} is an auxiliary step.

A \emph{cherry} is a path $y\,x\,y'$ of length two. We call $x$ its \emph{center} and $y,y'$ its \emph{ends}.

\begin{lemma}\label{lem:expansion}
Let $F$ be a graph and let $X$, $P$ and $Z$ be disjoint sets of vertices of $F$. Let $t,q\geq 1$ be real numbers and let $m\geq 20t$ be an integer. Suppose that every cycle of $F$ contains at most $t$ vertices of $X$, that each $x\in X$ has at least two neighbors in $P$, and that each $y\in P$ has at most $q$ neighbors in $X$ and exactly $m$ neighbors in $Z$. Then
$$m\lvert X\rvert\leq 20tq\lvert Z\rvert.$$
\end{lemma}

\begin{proof}

For each $x\in X$, we fix two neighbors $y_x,y'_x\in P$ and consider the cherry $y_x\,x\,y'_x$. We greedily choose a cherry and discard all cherries which share a vertex with it. Two of these cherries can share only an end since $X$ and $P$ are disjoint. Each vertex of $P$ is an end of at most $q$ cherries, so each choice discards at most $2q$ cherries including the chosen one. Hence we obtain a family $\mathcal{C}$ of at least $\lvert X\rvert/(2q)$ disjoint cherries. Let $Y$ be the set of their ends and let $r=2t+4$.

Suppose to the contrary that $m\lvert X\rvert>20tq\lvert Z\rvert$. Let $J$ be the bipartite graph formed by the edges between $Y$ and $Z$. Since $\lvert Y\rvert=2\lvert\mathcal{C}\rvert\geq\lvert X\rvert/q$ and each end has exactly $m$ neighbors in $Z$, the graph $J$ has $m\lvert Y\rvert\geq m\lvert X\rvert/q>20t\lvert Z\rvert\geq 3r\lvert Z\rvert$ edges, where we used $t\geq 1$. We remove vertices from $J$ until each remaining vertex has at least $r$ remaining neighbors. In each step, we remove a vertex of $Z$ with fewer than $r$ remaining neighbors, or an end with fewer than $r$ remaining neighbors. When we remove an end, we also remove the other end of its cherry. We remove fewer than $r\lvert Z\rvert$ edges together with vertices of $Z$. An end which is removed because it has fewer than $r$ remaining neighbors has lost at least $m-r$ edges by removals of vertices of $Z$, and distinct ends lose distinct edges. Thus at most $r\lvert Z\rvert/(m-r)$ ends are removed for this reason. With each of them, we remove fewer than $r$ edges, and with the other end of its cherry, at most $m$ edges. In total, we remove at most
$$r\lvert Z\rvert\Big(1+\frac{r+m}{m-r}\Big)\leq 3r\lvert Z\rvert$$
edges, where we used $m\geq 20t\geq 3r$. So some edges remain. Since ends are removed in pairs, both ends of each remaining cherry remain. Let $\mathcal{C}'$ be the set of remaining cherries and $Z'$ the set of remaining vertices of $Z$. Each end of a cherry of $\mathcal{C}'$ has at least $r$ neighbors in $Z'$, and each vertex of $Z'$ is adjacent to at least $r$ ends of cherries of $\mathcal{C}'$. As each cherry has two ends, each vertex of $Z'$ is adjacent to ends of at least $r/2=t+2$ cherries of $\mathcal{C}'$.

\begin{claim}\label{cl:cherry-cycle}
$F$ has a cycle through the centers of at least $t+1$ cherries of $\mathcal{C}'$.
\end{claim}

\begin{proof}[Proof of Claim~\ref{cl:cherry-cycle}]
Choose a longest sequence $z_1,K_1,z_2,K_2,\dots$ of distinct vertices $z_i\in Z'$ and distinct cherries $K_i\in\mathcal{C}'$ in which $z_i$ is adjacent to an end of $K_i$, its \emph{entry end}, and $z_{i+1}$ is adjacent to the other end of $K_i$, its \emph{exit end}. This sequence gives a path which passes through the cherries in turn (see Figure~\ref{fig:stars}).

Suppose first that the sequence ends with a cherry $K_k$. By maximality, each of the at least $r$ neighbors in $Z'$ of the exit end of $K_k$ occurs in the sequence. If $z_j$ is the first of them, then $j\leq k-r+1$, and the path from $z_j$ to the exit end of $K_k$ and the edge back to $z_j$ form a cycle through the centers of the at least $r$ cherries $K_j,\dots,K_k$, and $r\geq t+1$.

Suppose next that the sequence ends with a vertex $z_k$. By maximality, each of the at least $t+2$ cherries of $\mathcal{C}'$ with an end adjacent to $z_k$ occurs in the sequence. If $K_j$ is the first of them, then $j\leq k-t-2$. If $z_k$ is adjacent to the entry end of $K_j$, then it closes a cycle through the centers of $K_j,\dots,K_{k-1}$, and otherwise through the centers of $K_{j+1},\dots,K_{k-1}$.

In both cases, the cycle passes through the centers of at least $t+1$ cherries. This proves the claim.
\end{proof}

The cycle of Claim~\ref{cl:cherry-cycle} contains at least $t+1$ vertices of $X$, a contradiction.
\end{proof}

\begin{figure}[ht]
\centering
\begin{minipage}[c]{0.53\textwidth}
\centering
\begin{tikzpicture}[scale=0.9]
\useasboundingbox (-1.85,-1.4) rectangle (7.85,2.2);
\foreach \i/\x in {1/0,2/2.6,3/5.2}{
  \node[lowvtx] (x\i) at (\x+0.65,-0.95) {};
  \node[upvtx] (p\i) at (\x,0) {};
  \node[upvtx] (q\i) at (\x+1.3,0) {};
  \draw[thick] (p\i)--(x\i)--(q\i);
  \node[font=\footnotesize,below] at (x\i.south) {$x_{\i}$};
}
\foreach \i/\x in {1/-1.1,2/1.95,3/4.55,4/7.6}{
  \node[vtx] (z\i) at (\x,1.3) {};
  \node[font=\footnotesize,above] at (z\i.north) {$z_{\i}$};
}
\draw[thick,blue!60!black] (z1)--(p1); \draw[thick,blue!60!black] (q1)--(z2);
\draw[thick,blue!60!black] (z2)--(p2); \draw[thick,blue!60!black] (q2)--(z3);
\draw[thick,blue!60!black] (z3)--(p3); \draw[thick,blue!60!black] (q3)--(z4);
\draw[thick,red!70!black,dashed] (z4) .. controls (6.4,2.7) and (3.8,2.5) .. (p2);
\node[font=\small] at (-1.6,0.35) {$Z$};
\end{tikzpicture}
\end{minipage}\hfill
\begin{minipage}[c]{0.44\textwidth}
\setlength{\abovecaptionskip}{0pt}
\caption{A longest alternating sequence $z_1,K_1,\dots,K_3,z_4$ in the proof of Lemma~\ref{lem:expansion}, where the cherry $K_i$ has center $x_i$. The dashed edge from $z_4$ to the entry end of $K_2$ closes a cycle through $x_2$ and $x_3$.}
\label{fig:stars}
\end{minipage}
\end{figure}
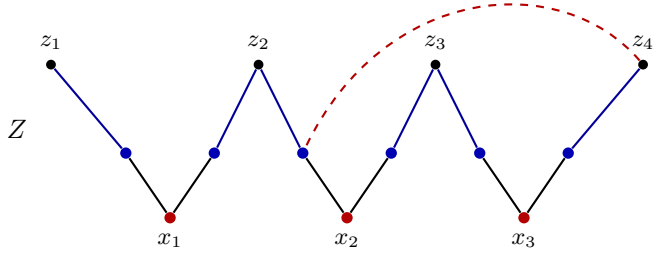

Unlike in Lemma~\ref{lem:expansion}, the condition on $P$ in the next lemma concerns the degrees in $F$, so the vertices of $P$ may have their neighbors anywhere, including in $X$ and $P$.

\begin{lemma}\label{lem:two-parents}
Let $F$ be a graph on at most $n$ vertices, and let $t,q\geq 1$ and $d\geq 1000t$ be real numbers. Let $X$ and $P$ be disjoint sets of vertices of $F$ such that every cycle of $F$ contains at most $t$ vertices of $X$. Suppose that each $x\in X$ has two distinct neighbors $y_x,y'_x\in P$, which we call the \emph{designated neighbors} of $x$, and that each $y\in P$ is a designated neighbor of at most $q$ vertices of $X$ and has degree at least $d$ in $F$. Then
$$\lvert X\rvert\leq\frac{2000tqn}{d}.$$
\end{lemma}

\begin{proof}
Let $m=\lceil d/4\rceil$. Then $m\geq d/4\geq 250t\geq 20t$. We put each vertex of $F$ into a set $W$ independently with probability $1/2$. Let $X'$ be the set of vertices $x\in X\cap W$ such that $y_x$ and $y'_x$ lie in $W$ and each of them has at least $m$ neighbors outside $W$.

Fix $x\in X$. The vertices $x$, $y_x$ and $y'_x$ are distinct, so they all lie in $W$ with probability $1/8$, and we condition on this event. Let $y\in\{y_x,y'_x\}$. The vertex $y$ has at least $d-2$ neighbors other than $x$, $y_x$ and $y'_x$, and each of them lies outside $W$ independently with probability $1/2$. So the number of these neighbors outside $W$ is a binomial random variable with mean $\mu\geq(d-2)/2$. As $d\geq 1000t\geq 1000$, we have $m\leq d/4+1\leq 2\mu/3$, so Chernoff's bound with $\theta=1/3$ shows that this number is less than $m$ with probability at most $\exp(-\mu/18)<1/4$. Hence $x\in X'$ with probability at least $\tfrac18(1-2\cdot\tfrac14)=\tfrac1{16}$.

We fix an outcome with $\lvert X'\rvert\geq\lvert X\rvert/16$. Let $P'$ be the set of designated neighbors of the vertices of $X'$, and let $Z=V(F)\setminus W$. Then $X'$, $P'$ and $Z$ are disjoint, and each $y\in P'$ has at least $m$ neighbors in $Z$. We apply Lemma~\ref{lem:expansion} to the subgraph of $F$ which consists of the edges $x\,y_x$ and $x\,y'_x$ for $x\in X'$ and of exactly $m$ edges from each vertex of $P'$ to $Z$. In this subgraph, every cycle contains at most $t$ vertices of $X'$, each $x\in X'$ has two neighbors in $P'$, and each $y\in P'$ has at most $q$ neighbors in $X'$ and exactly $m$ neighbors in $Z$. As $\lvert Z\rvert\leq n$ and $m\geq d/4$, this gives
$$\lvert X\rvert\leq 16\lvert X'\rvert\leq\frac{320tqn}{m}\leq\frac{1280tqn}{d}\leq\frac{2000tqn}{d}.$$
\end{proof}

We next apply Lemma~\ref{lem:two-parents} to vertices which have many neighbors of much smaller degree. For a vertex $y$ of a graph $G$, we call a neighbor $z$ of $y$ with $d(z)\leq d(y)^{1/3}$ a \emph{far-down neighbor} of $y$.

\begin{lemma}\label{lem:down}
Let $d$ be sufficiently large and let $r\geq 2$ be an integer. In every light $n$-vertex graph $G$, at most $(1/r+d^{-1/7})n$ vertices $y$ with $d(y)\geq d$ have at least $r$ far-down neighbors.
\end{lemma}

\begin{proof}
Let $U$ be the set of vertices $y$ with $d(y)\geq d$ which have at least $r$ far-down neighbors. For each vertex of $U$, we choose $r$ far-down neighbors of it. We split the vertices of small degree into dyadic degree classes $X_j$, and we apply Lemma~\ref{lem:two-parents} to each class to bound the number of vertices chosen at least twice.

For $j\geq 0$, put
$$b_j=2^j,\qquad t_j=b_j^{3/2}\qquad\text{and}\qquad d_j=\max\{d,(b_j/2)^3\},$$
and let $X_j$ be the set of vertices $z$ with $b_j/2<d(z)\leq b_j$.
If $y\in U$ has a far-down neighbor $z\in X_j$, then $d(y)\geq d_j$ since $d(y)\geq d$ and $d(y)\geq d(z)^3>(b_j/2)^3$. If $d$ is sufficiently large, then $d_j\geq 1000t_j$ and $d_j>b_j$ for all $j$. Indeed, for $b_j\leq 2d^{1/3}$, we have $d_j\geq d$ while $t_j\leq 3d^{1/2}$. For larger $b_j$, we have $(b_j/2)^3/t_j=b_j^{3/2}/8$, and this is large since $b_j>2d^{1/3}$. In particular, a vertex of $U$ with a far-down neighbor in $X_j$ does not lie in $X_j$.

Fix $j$ and let $X_{j,2}$ be the set of vertices of $X_j$ which are chosen at least twice. Let $P_j$ be the set of vertices of $U$ for which a vertex of $X_{j,2}$ is chosen. For each $x\in X_{j,2}$, we designate two distinct vertices of $U$ for which $x$ is chosen. Each $y\in P_j$ satisfies $d(y)\geq d_j$, and it is designated at most $r$ times since $r$ vertices are chosen for it.

We apply Lemma~\ref{lem:two-parents} to the light graph $G$, with $X_{j,2}$ and $P_j$ playing the roles of $X$ and $P$, with the designated vertices as designated neighbors, and with $t_j$, $d_j$ and $r$ playing the roles of $t$, $d$ and $q$. Its conditions hold. The sets $X_{j,2}$ and $P_j$ are disjoint, as shown above. Every cycle of $G$ contains fewer than $b_j^{3/2}=t_j$ vertices of $X_{j,2}$, since $G$ is light and $d(x)\leq b_j$ for $x\in X_{j,2}$. Finally, $d_j\geq 1000t_j$. So we obtain $\lvert X_{j,2}\rvert\leq 2000t_jrn/d_j$.

Each vertex of $X_j$ is chosen at most $b_j$ times since its degree is at most $b_j$, and the vertices which are chosen exactly once account for at most $n$ choices. Counting the choices gives
$$r\lvert U\rvert\leq n+\sum_{j\geq 0}b_j\lvert X_{j,2}\rvert\leq n+2000rn\sum_{j\geq 0}\frac{b_jt_j}{d_j}.$$
We have $b_jt_j/d_j=b_j^{5/2}/d_j$. The terms with $b_j\leq 2d^{1/3}$ have $d_j\geq d$, and as the numbers $b_j=2^j$ grow geometrically, they sum to $O((2d^{1/3})^{5/2}/d)=O(d^{-1/6})$. The terms with $b_j>2d^{1/3}$ have $d_j\geq b_j^3/8$, so they are at most $8b_j^{-1/2}$, and they sum to $O(d^{-1/6})$. Dividing by $r$, we obtain
$$\lvert U\rvert\leq\frac{n}{r}+O(d^{-1/6})n\leq\Big(\frac1r+d^{-1/7}\Big)n$$
for sufficiently large $d$.
\end{proof}

\section{Near-closed subgraphs of light graphs}\label{sec:structure}

In this section, we prove Theorem~\ref{thm:structure}. The subsections follow steps~\ref{st:window}--\ref{st:choose} of Section~\ref{sec:outline-structure}, and step~\ref{st:regions} occupies Sections~\ref{sec:graphF} and~\ref{sec:regions}. Let $G$ be a light graph on $n$ vertices and let $\delta=\delta(G)$. All degrees in this section are those of $G$. We assume that $\delta$ is sufficiently large. All asymptotic statements below refer to $\delta\to\infty$, and their implicit constants are absolute. Since the number $a$ chosen below satisfies $a\geq\delta$, every $o(1)$ which depends on $a$ also tends to zero as $\delta\to\infty$.

\subsection{Choosing a degree window}\label{sec:window}

We begin by choosing a window $L$ with top $a$ such that $\lvert L\rvert\geq n/(1000\log a)$ and at least $\lvert L\rvert/11$ vertices of $L$ lie in the middle of the window and have at most one neighbor of much smaller degree (see~\eqref{eq:population} and~\eqref{eq:one-low}).

Let $\sigma\in\{0,1,2,3,4\}$. For $j\geq 0$, let $a_j=\delta^{3^{\sigma+5j}}$. The \emph{windows} are the sets
$$L_0=\{v:d(v)\leq a_0\}\qquad\text{and}\qquad L_j=\{v:a_{j-1}<d(v)\leq a_j\}\quad(j\geq 1).$$
The numbers $a_j$ and the windows depend on $\sigma$, which we choose in Claim~\ref{cl:window}, but we do not write $\sigma$ in the notation. We have $a_{j-1}=a_j^{1/243}$ for $j\geq 1$, so for each $\sigma$ the windows partition $V(G)$.

Recall that a far-down neighbor of a vertex $v$ is a neighbor of degree at most $d(v)^{1/3}$. Let $A^*$ be the set of vertices which have at most one far-down neighbor. By Lemma~\ref{lem:down} with $\delta$ and $2$ playing the roles of $d$ and $r$, we have $\lvert A^*\rvert\geq(1/2-\delta^{-1/7})n$. For $j\geq 0$, let
$$A_j=\{v\in A^*:a_j^{1/27}<d(v)\leq a_j^{1/9}\}.$$
Then $A_j\subseteq L_j$, since $a_j^{1/27}>a_j^{1/243}$.
\begin{claim}\label{cl:window}
There are $\sigma$ and $j$ such that $\lvert A_j\rvert\geq\lvert L_j\rvert/11$ and $\lvert L_j\rvert\geq n/(1000\log a_j)$.
\end{claim}

\begin{proof}[Proof of Claim~\ref{cl:window}]
Since $d(v)\geq\delta$, each vertex $v$ satisfies $\delta^{3^{i-3}}<d(v)\leq\delta^{3^{i-2}}$ for exactly one integer $i\geq 2$. As $a_j^{1/27}=\delta^{3^{\sigma+5j-3}}$ and $a_j^{1/9}=\delta^{3^{\sigma+5j-2}}$, writing $i=\sigma+5j$ with $\sigma\in\{0,\dots,4\}$ shows that each $v\in A^*$ lies in $A_j$ for exactly one pair $(\sigma,j)$. So the five sums $\sum_j\lvert A_j\rvert$ for $\sigma=0,\dots,4$ add up to $\lvert A^*\rvert$, and we fix $\sigma$ with $\sum_j\lvert A_j\rvert\geq\lvert A^*\rvert/5\geq n/10-\delta^{-1/7}n$.

The windows $L_j$ with $\lvert L_j\rvert<n/(1000\log a_j)$ contain fewer than $n/500$ vertices in total, as $\log a_j\geq 3^{5j}$. As $A_j\subseteq L_j$, the sum of $\lvert A_j\rvert$ over the other windows is at least $n/10-\delta^{-1/7}n-n/500\geq n/11$ for large $\delta$, while their sizes add up to at most $n$. So one of these windows has $\lvert A_j\rvert\geq\lvert L_j\rvert/11$. This proves the claim.
\end{proof}

We fix such $\sigma$ and $j$, and we write
$$L=L_j,\qquad A=A_j,\qquad \lambda=\lvert L\rvert\qquad\text{and}\qquad a=a_j.$$
By Claim~\ref{cl:window} and the definition of $A_j$, we obtain
\begin{equation}\label{eq:population}
\lvert A\rvert\geq\frac{\lambda}{11},\qquad n\leq 1000\lambda\log a\qquad\text{and}\qquad a^{1/27}<d(v)\leq a^{1/9}\ \ (v\in A).
\end{equation}

Let
$$d_0=a^{1/81},\qquad d_1=a^{1/3}\qquad\text{and}\qquad I=\{v:d_0\leq d(v)\leq d_1\}.$$
If $j\geq 1$, then $a_{j-1}=a^{1/243}$, so $L$ consists of the vertices $v$ with $a^{1/243}<d(v)\leq a$. If $j=0$, then $L$ consists of all vertices $v$ with $d(v)\leq a$, and $a^{1/243}<\delta$ since $a=\delta^{3^\sigma}$ with $\sigma\leq 4$. In both cases, each vertex $v$ with $a^{1/243}<d(v)\leq a$ lies in $L$. In particular, each vertex of $I$ lies in $L$ since $d_0>a^{1/243}$, and we call $I$ the \emph{interior} of the window. We have $A\subseteq I$. For $x\in A$, we have $d(x)^{1/3}\geq a^{1/81}=d_0$, so each neighbor of $x$ of degree less than $d_0$ has degree at most $d(x)^{1/3}$. As $x\in A^*$, we obtain the following.
\begin{equation}\label{eq:one-low}
\text{Each vertex of $A$ has at most one neighbor of degree less than $d_0$.}
\end{equation}

The window is divided into five ranges, each spanning a factor $3$ in the exponent (Figure~\ref{fig:window}). The degrees of the vertices of $A$ lie in the middle range, and each of the other ranges has its own role. The range between $a^{1/243}$ and $d_0$ ensures that only $o(\lambda)$ vertices of $I$ are removed in Section~\ref{sec:peeling}, since at most $a^{1/243}n$ edges join $I$ to vertices outside $L\cup U$, where $U$ is the set defined in Section~\ref{sec:graphF} (see~\eqref{eq:outside}). The range between $d_0$ and $a^{1/27}$ gives~\eqref{eq:one-low}. The range between $a^{1/9}$ and $d_1$ is used in Claim~\ref{cl:C2}, and the range between $d_1$ and $a$ is used in Claim~\ref{cl:U}.

\begin{figure}[ht]
\centering
\begin{tikzpicture}[x=1.6cm,y=1cm]
\draw[thick,-{Stealth}] (-0.3,0) -- (6.4,0) node[right,font=\small]{$\log\log d(v)$};
\foreach \x/\lab in {0/{a^{1/243}},1/{a^{1/81}},2/{a^{1/27}},3/{a^{1/9}},4/{a^{1/3}},5/{a}}{
  \draw (\x,0.08)--(\x,-0.08) node[below,font=\footnotesize]{$\lab$};
}
\fill[blue!12] (1,0.15) rectangle (4,1.25);
\fill[red!25] (2,0.15) rectangle (3,0.9);
\draw[blue!60!black,thick] (1,0.15) rectangle (4,1.25);
\draw[red!70!black,thick] (2,0.15) rectangle (3,0.9);
\node[font=\footnotesize,blue!50!black] at (1.5,1.0) {$I$};
\node[font=\footnotesize,red!60!black] at (2.5,0.55) {$A$};
\draw[decorate,decoration={brace,amplitude=5pt}] (0,1.5) -- (5,1.5) node[midway,above=5pt,font=\small]{window $L$ with top $a$};
\node[font=\footnotesize,align=center] at (-0.55,0.8) {degree\\ $\leq a^{1/243}$};
\node[font=\footnotesize,align=center] at (5.9,0.8) {degree\\ $>a$};
\node[font=\footnotesize] at (1,-0.75) {$d_0$};
\node[font=\footnotesize] at (4,-0.75) {$d_1$};
\end{tikzpicture}
\caption{One degree window on a doubly logarithmic scale, with the degree ranges of $I$ and $A$.}
\label{fig:window}
\end{figure}
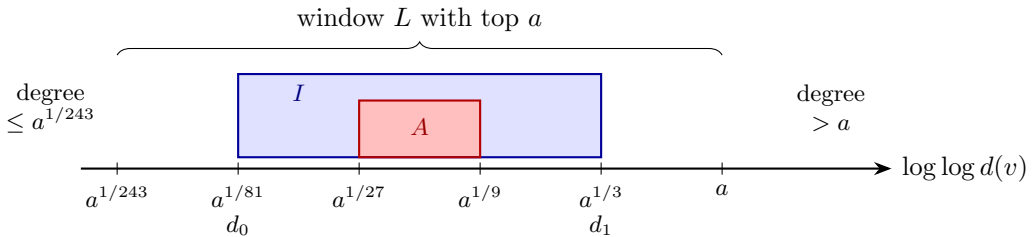

\subsection{The set \texorpdfstring{$U$}{U} and the graph \texorpdfstring{$F$}{F}}\label{sec:graphF}

Next, we choose a set $U$ of $o(\lambda)$ vertices of degree greater than $a$ and a graph $F$ on $L\cup U$ in which $U$ is independent, such that at most $a^{1/243}n$ edges join the interior $I$ to vertices outside $L\cup U$ (see~\eqref{eq:outside}).

We put $\tau=a^{1/300}$. This is the minimum degree which we obtain in the end. Let $U$ be the set of vertices $v$ with $d(v)>a$ which have at least $\tau$ neighbors in $I$, and let
$$P_0=\{v:d(v)>a\}\setminus U.$$
We use the threshold $\tau$ so that the vertices of $P_0$ and the vertices of large degree removed in Section~\ref{sec:peeling} both have fewer than $\tau$ neighbors in the set $C$ chosen there (Claim~\ref{cl:union}). Then Claim~\ref{cl:C2} treats them together, in one application of Lemma~\ref{lem:two-parents}, so that it also covers vertices of $C$ with one neighbor of each kind.

\begin{claim}\label{cl:U}
$\lvert U\rvert=o(\lambda)$.
\end{claim}

\begin{proof}[Proof of Claim~\ref{cl:U}]
Let $y\in U$. Each vertex of $I$ has degree at most $d_1=a^{1/3}<d(y)^{1/3}$, so $y$ has at least $\tau$ far-down neighbors. By Lemma~\ref{lem:down} with $a$ and $\tau$ playing the roles of $d$ and $r$, and since $n\leq 1000\lambda\log a$, we obtain
$$\lvert U\rvert\leq\Big(\frac{1}{\tau}+a^{-1/7}\Big)n\leq 1000\big(a^{-1/300}+a^{-1/7}\big)\lambda\log a =o(\lambda).$$
This proves the claim.
\end{proof}

Let $F$ be the graph with vertex set $L\cup U$ whose edges are the edges of $G$ with both ends in $L\cup U$ except the edges with both ends in $U$. We leave out the edges inside $U$ so that the vertices of $F$ of $G$-degree greater than $a$, which lie in $U$, form an independent set, as Lemma~\ref{lem:light-circumference} requires. We put $U$ into $F$ since vertices of $I$ may have many neighbors in $U$. We leave out $P_0$, which may be large.

Each vertex of degree at most $a$ outside $L$ has degree at most $a^{1/243}$, and each vertex of degree greater than $a$ outside $U$ lies in $P_0$ and so has fewer than $\tau=a^{1/300}<a^{1/243}$ neighbors in $I$. So each vertex outside $L\cup U$ has at most $a^{1/243}$ neighbors in $I$, and as $n\leq 1000\lambda\log a$ by~\eqref{eq:population}, we obtain
\begin{equation}\label{eq:outside}
e_G(I,V(G)\setminus(L\cup U))\leq a^{1/243}n\leq 1000a^{1/243}\lambda\log a.
\end{equation}

\subsection{Regions}\label{sec:regions}

The goal of this subsection is to find disjoint regions $R_i$ in $F$ with small separators $S_i$, such that each region has $O(a^2)$ vertices and the regions together contain all but $o(\lambda)$ vertices of $F$.

We put
$$k=2a^{3/2}\qquad\text{and}\qquad m=a^2.$$
The vertices of $F$ of $G$-degree greater than $a$ lie in the independent set $U$, so by Lemma~\ref{lem:light-circumference} with $b=a$, every cycle of $F$ has length less than $k$. For large $\delta$, the following estimates hold.
\begin{equation}\label{eq:ratios}
\frac{\tau^2}{d_0}=o(1),\qquad \frac{\tau a^{1/243}\log a}{d_0}=o(1)\qquad\text{and}\qquad \frac{\tau k}{m}=o(1).
\end{equation}
These estimates hold as $\tau=a^{1/300}$ and $d_0=a^{1/81}$.

We apply Lemma~\ref{lem:regions} to $F$, which is possible as $k\geq 3$. This gives regions $R_1,\dots,R_p$ and separators $S_1,\dots,S_p$. We write $F_i=F[R_i\cup S_i]$. Since $\tau\geq 1$, \eqref{eq:ratios} gives $k/m=o(1)$. As $\lvert F\rvert\leq\lambda+\lvert U\rvert=(1+o(1))\lambda$ by Claim~\ref{cl:U}, we have $\sum_i\lvert S_i\rvert\leq 4k\lambda/m=o(\lambda)$ and
\begin{equation}\label{eq:regions}
\Big\lvert V(F)\setminus\bigcup_{i\in[p]}R_i\Big\rvert=o(\lambda),\qquad\lvert R_i\cup S_i\rvert<2m+2k\qquad\text{and}\qquad\sum_{i\in[p]}\lvert F_i\rvert\leq\lvert F\rvert+\sum_{i\in[p]}\lvert S_i\rvert\leq(1+o(1))\lambda.
\end{equation}
Each edge of $G$ between $L$ and $L\cup U$ lies in $F$, and $N_F(R_i)\subseteq S_i$. Hence the following holds for each $i\in[p]$.
\begin{equation}\label{eq:region-edges}
\text{If $x\in L$ and $y\in L\cup U$ are adjacent and one of them lies in $R_i$, then $xy$ is an edge of $F_i$.}
\end{equation}

\subsection{Removing vertices of small degree in the regions}\label{sec:peeling}

We now remove vertices of small degree in each region, so that the remaining set $V_i$ induces a subgraph of minimum degree at least $\tau$. We then choose a set $C$ which contains all but $o(\lambda)$ of the vertices of $A$ in the union of the regions (Claim~\ref{cl:C}).

In each graph $F_i$, we repeatedly remove a vertex of current degree less than $\tau$ until no such vertex remains. Let $V_i$ be the set of remaining vertices, and let $\overline{V}_i=V(F_i)\setminus V_i$ be the set of vertices removed from $F_i$. If $V_i$ is nonempty, then $F[V_i]$ has minimum degree at least $\tau$, and so does $G[V_i]$. When a vertex of $\overline{V}_i$ is removed, it has fewer than $\tau$ neighbors in $F_i$ which are removed later or lie in $V_i$. So $e(F_i[\overline{V}_i])+e_{F_i}(\overline{V}_i,V_i)\leq\tau\lvert\overline{V}_i\rvert$, and
\begin{equation}\label{eq:removed}
\text{each vertex of $\overline{V}_i$ has fewer than $\tau$ neighbors in $V_i$ in the graph $F_i$.}
\end{equation}

Let $Q_i=I\cap R_i\cap\overline{V}_i$, and let $v\in Q_i$. As $v\in L$, each edge of $G$ from $v$ to $L\cup U$ lies in $F_i$ by~\eqref{eq:region-edges}, and the other edges at $v$ join $v$ to $V(G)\setminus(L\cup U)$. Summing the degrees in $F_i$ of the vertices of $Q_i$ gives
$$\sum_{v\in Q_i}d_{F_i}(v)\leq 2e(F_i[\overline{V}_i])+e_{F_i}(\overline{V}_i,V_i)\leq 2\tau\lvert\overline{V}_i\rvert\leq 2\tau\lvert F_i\rvert.$$
Each vertex of $Q_i$ has $G$-degree at least $d_0$, and the sets $Q_i$ are disjoint subsets of $I$. As $\sum_i\lvert F_i\rvert\leq 2\lambda$ by~\eqref{eq:regions}, we obtain by~\eqref{eq:outside}
$$d_0\sum_{i\in[p]}\lvert Q_i\rvert\leq\sum_{i\in[p]}2\tau\lvert F_i\rvert+e_G(I,V(G)\setminus(L\cup U))\leq 4\tau\lambda+1000a^{1/243}\lambda\log a.$$
Multiplying by $\tau/d_0$, the first two estimates of~\eqref{eq:ratios} give $\tau\sum_i\lvert Q_i\rvert=o(\lambda)$, and so $\sum_i\lvert Q_i\rvert=o(\lambda)$ as $\tau\geq 1$.

Let $C$ be the set of vertices $x\in A$ which lie in some region $R_i$ and in its set $V_i$, and which satisfy the following two conditions.
\begin{enumerate}[label=\upshape(D\arabic*)]
\item\label{ex:sep} $x$ lies in none of the separators $S_1,\dots,S_p$.
\item\label{ex:adj} $x$ has no neighbor in $\overline{V}_i\cap(S_i\cup I)$.
\end{enumerate}
\begin{claim}\label{cl:C}
All but $o(\lambda)$ of the vertices of $A$ in the union of the regions belong to $C$.
\end{claim}

\begin{proof}[Proof of Claim~\ref{cl:C}]
Let $x\in A\cap R_i$ with $x\notin C$. If $x\notin V_i$, then $x\in Q_i$ since $A\subseteq I$. Otherwise, $x$ fails~\ref{ex:sep} or~\ref{ex:adj}. At most $\sum_i\lvert S_i\rvert\leq 4k\lambda/m$ vertices fail~\ref{ex:sep}. Let $y$ be a vertex of $\overline{V}_i\cap(S_i\cup I)$. Then $y\in L\cup U$, as $\overline{V}_i\subseteq V(F)=L\cup U$. If $x\in A\cap R_i\cap V_i$ is adjacent to $y$ in $G$, then $xy$ is an edge of $F_i$ by~\eqref{eq:region-edges}, as $A\subseteq L$. So $y$ has fewer than $\tau$ such neighbors $x$ by~\eqref{eq:removed}. The vertices of $\overline{V}_i\cap I$ lie in $S_i$ or in $Q_i$. So at most $\tau\big(\sum_i\lvert S_i\rvert+\sum_i\lvert Q_i\rvert\big)$ vertices fail~\ref{ex:adj}. As $\tau\sum_i\lvert Q_i\rvert=o(\lambda)$ and by~\eqref{eq:ratios}, all these bounds are $o(\lambda)$. This proves the claim.
\end{proof}

\subsection{Neighbors outside the sets \texorpdfstring{$V_i$}{V\_i}}\label{sec:parentcount}

Our next aim is to show that all but $o(\lambda)$ vertices $x\in C$ have at most two neighbors outside $V_i$, where $R_i$ is the region containing $x$ (Claims~\ref{cl:C2} and~\ref{cl:two-outside}). We define a set $P_1$ of removed vertices of large degree, show that the vertices of $P_0\cup P_1$ have fewer than $\tau$ neighbors in $C$, and apply Lemma~\ref{lem:two-parents} with $P=P_0\cup P_1$.

Let $P_1$ be the set of vertices $v$ with $d(v)>d_1=a^{1/3}$ which lie in $R_i\cap\overline{V}_i$ for some $i$.

\begin{claim}\label{cl:union}
$\lvert N(y)\cap C\rvert<\tau$ for each $y\in P_0\cup P_1$.
\end{claim}

\begin{proof}[Proof of Claim~\ref{cl:union}]
If $y\in P_0$, this holds since $y$ has fewer than $\tau$ neighbors in $I\supseteq C$. Suppose $y\in P_1\cap R_i$ for some $i$. Let $x\in N(y)\cap C$. As $x\in L$ and $y\in R_i\subseteq L\cup U$, the edge $xy$ lies in $F_i$ by~\eqref{eq:region-edges}, so $x\in R_i\cup S_i$. By~\ref{ex:sep}, the vertex $x\in C$ does not lie in $S_i$, so $x\in R_i$. As $x\in C$ and the regions are disjoint, it lies in $V_i$. So $x$ is a neighbor of $y$ in $V_i$ in the graph $F_i$, and $y$ has fewer than $\tau$ such neighbors by~\eqref{eq:removed}. This proves the claim.
\end{proof}

Figure~\ref{fig:regions} illustrates~\ref{ex:adj} and Claim~\ref{cl:union}. Let $C_2$ be the set of vertices of $C$ with at least two neighbors in $P_0\cup P_1$.

\begin{claim}\label{cl:C2}
$\lvert C_2\rvert=o(\lambda)$.
\end{claim}

\begin{proof}[Proof of Claim~\ref{cl:C2}]
Let $t=a^{1/6}$. We apply Lemma~\ref{lem:two-parents} to the light graph $G$ with $C_2$, $P_0\cup P_1$, $\tau$ and $d_1$ playing the roles of $X$, $P$, $q$ and $d$. The sets $C_2$ and $P_0\cup P_1$ are disjoint since $a^{1/9}<d_1$. Each vertex of $C_2$ has degree at most $a^{1/9}$ by~\eqref{eq:population}, so every cycle of $G$ contains fewer than $(a^{1/9})^{3/2}=t$ vertices of $C_2$, as $G$ is light. We designate two neighbors in $P_0\cup P_1$ of each vertex of $C_2$. Each vertex of $P_0\cup P_1$ has degree greater than $d_1$. By Claim~\ref{cl:union}, it has fewer than $\tau$ neighbors in $C_2$, so it is designated fewer than $\tau$ times. The condition $d_1\geq 1000t$ holds for large $\delta$. As $n\leq 1000\lambda\log a$, we obtain
$$\lvert C_2\rvert\leq\frac{2000t\tau n}{d_1}\leq 2\cdot 10^6a^{1/6+1/300-1/3}\lambda\log a =o(\lambda).$$
This proves the claim.
\end{proof}

We discard $C_2$ from $C$ and write $C'=C\setminus C_2$.

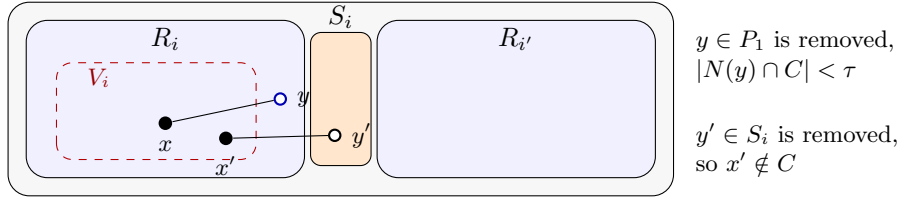
\begin{figure}[ht]
\centering
\begin{tikzpicture}[scale=0.8]
\draw[rounded corners=8pt,fill=gray!6] (0,0) rectangle (11,3.2);
\draw[rounded corners=8pt,fill=blue!6] (0.3,0.3) rectangle (4.9,2.9);
\draw[rounded corners=8pt,fill=blue!6] (6.1,0.3) rectangle (10.7,2.9);
\node[font=\small] at (2.6,2.6) {$R_i$};
\node[font=\small] at (8.4,2.6) {$R_{i'}$};
\draw[rounded corners=4pt,fill=orange!20] (5.0,0.5) rectangle (6.0,2.7);
\node[font=\small] at (5.5,2.95) {$S_i$};
\draw[dashed,red!70!black,rounded corners=6pt] (0.8,0.6) rectangle (4.1,2.2);
\node[font=\footnotesize,red!60!black] at (1.5,1.95) {$V_i$};
\node[bigvtx,label={[font=\footnotesize]below:$x$}] (x) at (2.6,1.2) {};
\node[circle,draw=blue!70!black,thick,fill=white,inner sep=1.5pt,label={[font=\footnotesize]right:$y$}] (p) at (4.5,1.6) {};
\node[circle,draw=black,thick,fill=white,inner sep=1.5pt,label={[font=\footnotesize]right:$y'$}] (s) at (5.4,1.0) {};
\node[bigvtx,label={[font=\footnotesize]below:$x'$}] (x2) at (3.6,0.95) {};
\draw (x)--(p); \draw (x2)--(s);
\node[font=\footnotesize,align=left,anchor=west] at (11.2,2.3) {$y\in P_1$ is removed,\\ $\lvert N(y)\cap C\rvert<\tau$};
\node[font=\footnotesize,align=left,anchor=west] at (11.2,0.8) {$y'\in S_i$ is removed,\\ so $x'\notin C$};
\end{tikzpicture}
\caption{Regions after the removal of vertices of small degree. Open circles are removed vertices, and $x\in C$.}
\label{fig:regions}
\end{figure}

\begin{claim}\label{cl:two-outside}
For each $i$, each vertex $x\in C'\cap R_i$ has at most two neighbors outside $V_i$.
\end{claim}

\begin{proof}[Proof of Claim~\ref{cl:two-outside}]
See Figure~\ref{fig:steps}(4). Note that $x\in V_i$ by the definition of $C$. Let $y$ be a neighbor of $x$ outside $V_i$ with $d(y)\geq d_0$ and $y\notin P_0$. If $d(y)\leq a$, then $y\in L$ as $d(y)\geq d_0=a^{1/81}>a^{1/243}$, and otherwise $y\in U$ as $y\notin P_0$. So $xy$ is an edge of $F_i$ by~\eqref{eq:region-edges} and $y\in R_i\cup S_i$. As $y\notin V_i$, we have $y\in\overline{V}_i$. By~\ref{ex:adj}, the vertex $y$ lies neither in $S_i$ nor in $I$. So $y\in R_i$, and $d(y)>d_1$ as $d(y)\geq d_0$ and $y\notin I$. Hence $y\in P_1$.

Thus each neighbor of $x$ outside $V_i$ has degree less than $d_0$ or lies in $P_0\cup P_1$. By~\eqref{eq:one-low}, at most one neighbor of $x$ has degree less than $d_0$, and as $x\notin C_2$, at most one lies in $P_0\cup P_1$. This proves the claim.
\end{proof}

\subsection{Choosing one region}\label{sec:choose}

Finally, we choose one region $R_i$ and show that $S=V_i$ has the properties of Theorem~\ref{thm:structure}. Recall that $k=2a^{3/2}$ and $m=a^2$.

By~\eqref{eq:population} and~\eqref{eq:regions}, at least $\lambda/11-o(\lambda)$ vertices of $A$ lie in the union of the regions. By Claims~\ref{cl:C} and~\ref{cl:C2}, all but $o(\lambda)$ of them lie in $C'$, so $\lvert C'\rvert\geq\lambda/11-o(\lambda)$.

As $C'$ lies in the union of the regions and $\lvert C'\rvert\geq\lambda/11-o(\lambda)\geq\sum_i\lvert F_i\rvert/12$ for large $\delta$ by~\eqref{eq:regions}, some region $R_i$ which contains a vertex of $C'$ satisfies $\lvert C'\cap R_i\rvert\geq\lvert F_i\rvert/12$. We put $S=V_i$. Each vertex of $C'\cap R_i$ lies in $V_i$, so $S$ is nonempty and
$$\lvert S\rvert\leq\lvert F_i\rvert<2m+2k\leq a^3.$$
By Section~\ref{sec:peeling}, $G[S]$ has minimum degree at least $\tau=a^{1/300}$, and $\tau\geq\lvert S\rvert^{1/1000}$ as $\lvert S\rvert<a^3$. Moreover, at least $\lvert C'\cap R_i\rvert\geq\lvert F_i\rvert/12\geq\lvert S\rvert/12$ vertices of $S$ have at most two neighbors outside $S$ by Claim~\ref{cl:two-outside}. This completes the proof of Theorem~\ref{thm:structure}.

\section{Graphs of polynomial minimum degree}\label{sec:qoh}

In this section, we prove Theorem~\ref{thm:qoh}. We first define robust sublinear expanders and state the results of Buci\'c and Montgomery and of Conlon, Fox and Sudakov that we use.

The following notion of robust sublinear expansion is a special case of a notion of Buci\'c and Montgomery~\cite[Definition~11]{bm}. It is a robust form of the sublinear expansion of Koml\'os and Szemer\'edi~\cite{ks94}, which has found many applications (see, for example,~\cite{gfl,letzter,lm17,lm23}). Similar robust notions were used by Haslegrave, Kim and Liu~\cite{hkl} and by Sudakov and Tomon~\cite{st}.

\begin{definition}\label{def:expander}
Let $s\geq 0$. A graph $J$ on $m\geq 2$ vertices is an \emph{$s$-expander} if
$$\lvert N_{J-E'}(X)\rvert\geq\frac{\lvert X\rvert}{32\log^2 m}$$
for every $X\subseteq V(J)$ with $1\leq\lvert X\rvert\leq 2m/3$ and every $E'\subseteq E(J)$ with $\lvert E'\rvert\leq s\lvert X\rvert$.
\end{definition}

Every $s$-expander has minimum degree greater than $s$, as otherwise $X=\{v\}$ for a vertex $v$ of degree at most $s$ and the set $E'$ of edges at $v$ violate Definition~\ref{def:expander}.

Statements~\ref{bm:paths} and~\ref{bm:split} below follow from~\cite[Corollary~22]{bm} and~\cite[Lemma~14]{bm}, and statement~\ref{bm:bulk} is~\cite[Theorem~24]{bm}.

\begin{theorem}[Buci\'c and Montgomery~\cite{bm}]\label{thm:bm}
The following statements hold.
\begin{enumerate}[label=\upshape(\roman*)]
\item\label{bm:paths} Every graph on $m$ vertices can be decomposed into at most $m$ paths.
\item\label{bm:split} Let $s\geq 0$ be an integer. For every graph $M$ on $m$ vertices, we can remove at most $4sm\log m$ edges of $M$ so that the remaining edges can be partitioned into $s$-expanders whose orders sum to at most $2m$.
\item\label{bm:bulk} There is an absolute constant $\beta$ such that every graph on $m\geq 2$ vertices can be decomposed into at most $6m$ cycles and at most $\beta m\log^{274}m$ edges.
\end{enumerate}
\end{theorem}

In~\cite[Lemma~14]{bm}, a graph with one vertex is also regarded as an expander. Such graphs have no edges, so statement~\ref{bm:split} does not include them.

A graph is \emph{$(\rho,c)$-sparse} if every set $X$ of vertices spans at most $c\lvert X\rvert^{2-\rho}$ edges. This property is inherited by subgraphs. Statements~\ref{cfs:dec} and~\ref{cfs:random} below are~\cite[Corollary~4.3]{cfs} and~\cite[Lemma~4.4]{cfs}.

\begin{theorem}[Conlon, Fox and Sudakov~\cite{cfs}]\label{thm:cfs}
The following statements hold.
\begin{enumerate}[label=\upshape(\roman*)]
\item\label{cfs:dec} For $0<\rho<1\leq c$, there is $b=b(\rho,c)$ such that every $(\rho,c)$-sparse graph $M$ on $m$ vertices can be decomposed into at most $bm$ cycles and edges.
\item\label{cfs:random} Let $0<\rho<1/2$. With probability $1-o(1)$ as $m\to\infty$, uniformly in $p\leq m^{-\rho}$, the random graph $G(m,p)$ is $(\rho,2/\rho)$-sparse.
\end{enumerate}
\end{theorem}

\subsection{Random subgraphs and routing}\label{sec:routing}

We next prove three lemmas on robust expanders, one on random subgraphs, one on short paths and one on routing. Lemma~\ref{lem:split} is a variant of~\cite[Lemma~15]{bm} with a short direct proof. Lemma~\ref{lem:core} keeps a robust core of an expander after the removal of some edges, an idea used by Alon and Capalbo~\cite{ac} for disjoint paths in expanders.

\begin{lemma}\label{lem:split}
Let $J$ be an $m$-vertex $s$-expander, and let $h$ be a positive integer with $s\geq 64h\log m$. If we keep each edge of $J$ independently with probability $1/h$, then the resulting spanning subgraph is an $s/(2h)$-expander with probability at least $1-m^{-5}$.
\end{lemma}

\begin{proof}
Let $J'$ be the random subgraph, and write $\ell=\log m$. Let $U$ and $W$ be disjoint sets of vertices of $J$ with $1\leq\lvert U\rvert\leq 2m/3$ and $\lvert W\rvert<\lvert U\rvert/(32\ell^2)$. Then more than $s\lvert U\rvert$ edges of $J$ join $U$ to vertices outside $U\cup W$, since otherwise the removal of these edges would leave all neighbors of $U$ in $W$, contrary to the expansion of $J$. We call the pair $(U,W)$ \emph{bad} if at most $s\lvert U\rvert/(2h)$ of these edges lie in $J'$. By Chernoff's bound with $\theta=1/2$, this happens with probability at most $\exp(-s\lvert U\rvert/(8h))\leq\exp(-8\ell\lvert U\rvert)\leq m^{-8\lvert U\rvert}$, where we used $s\geq 64h\ell$ and $e^{\ell}\geq 2^{\ell}=m$.

For $u\geq 1$, there are at most $m^u$ sets $U$ of size $u$ and at most $m^u$ sets $W$ of size less than $u/(32\ell^2)\leq u$. So the probability that some pair is bad is at most $\sum_{u\geq 1}m^{2u}m^{-8u}\leq 2m^{-6}\leq m^{-5}$. Suppose that no pair is bad, and let $U\subseteq V(J)$ with $1\leq\lvert U\rvert\leq 2m/3$ and $E''\subseteq E(J')$ with $\lvert E''\rvert\leq s\lvert U\rvert/(2h)$. Let $W=N_{J'-E''}(U)$. If $\lvert W\rvert<\lvert U\rvert/(32\ell^2)$, then all edges of $J'$ from $U$ to vertices outside $U\cup W$ lie in $E''$, so there are at most $s\lvert U\rvert/(2h)$ of them, and the pair $(U,W)$ is bad. So $\lvert N_{J'-E''}(U)\rvert\geq\lvert U\rvert/(32\ell^2)$, and $J'$ is an $s/(2h)$-expander.
\end{proof}

\begin{lemma}\label{lem:core}
Let $L$ be a positive integer and put $r=2^8L^3$. Let $J$ be an $s$-expander with $s>0$ on $m\leq 2^L$ vertices, and let $F\subseteq E(J)$ with $\lvert F\rvert\leq sm/12$. Then there is a set $D\subseteq V(J)$ with $\lvert D\rvert\leq 2\lvert F\rvert/s$ such that for every $Q\subseteq V(J)\setminus D$ with $\lvert Q\rvert\leq s/2$, any two distinct vertices of $(J-F)-D-Q$ are joined by a path of length at most $r$ in $(J-F)-D-Q$.
\end{lemma}

\begin{proof}
Put $\eta=1/(32L^2)$. As $\log m\leq L$, we have $\lvert N_{J-E'}(X)\rvert\geq\eta\lvert X\rvert$ for all $X$ and $E'$ as in Definition~\ref{def:expander}. For a set $D$, we write $K_D=(J-F)-D$.

We remove from $J-F$, one at a time, sets which expand poorly after the deletion of a few vertices. We start with $D=\varnothing$ and repeat the following step. If there are a set $Q\subseteq V(K_D)$ with $\lvert Q\rvert\leq s/2$ and a nonempty set $U\subseteq V(K_D)\setminus Q$ with $\lvert U\rvert\leq\lvert K_D\rvert/2$ such that $T_U=N_{K_D-Q}(U)$ has fewer than $\eta\lvert U\rvert/2$ vertices, we add $U$ to $D$ and write $Q_U=Q$. Otherwise we stop. We show that $\lvert D\rvert\leq 2\lvert F\rvert/s$ throughout, which gives $\lvert D\rvert\leq m/6$ as $\lvert F\rvert\leq sm/12$. This holds for $D=\varnothing$. Suppose that it holds before an addition. The added set has at most $(m-\lvert D\rvert)/2$ vertices, so afterwards $\lvert D\rvert\leq(m+m/6)/2<2m/3$. Let $T$ be the set of vertices outside $D$ which lie in one of the sets $T_U$ chosen so far. Then $\lvert T\rvert<\eta\lvert D\rvert/2$, as the added sets are disjoint. Let $xy$ be an edge of $J$ with $x\in D$ and $y\notin D\cup T$, and let $U$ be the added set which contains $x$. When $U$ was added to the set $D'$ of that time, $y$ was a vertex of $K_{D'}$ outside $U\cup T_U$, so $xy\in F$ or $y\in Q_U$. As $\lvert Q_U\rvert\leq s/2$, at most $\lvert F\rvert+s\lvert D\rvert/2$ edges of $J$ join $D$ to vertices outside $D\cup T$, and their removal leaves all neighbors of $D$ in $T$. By the expansion of $J$, there are more than $s\lvert D\rvert$ of these edges, so $\lvert D\rvert<2\lvert F\rvert/s$. As $D$ grows with each addition, the process stops.

Let $K=K_D$ for the final set $D$, and let $Q\subseteq V(K)$ with $\lvert Q\rvert\leq s/2$. As the process stopped, for every nonempty $U\subseteq V(K)\setminus Q$ with $\lvert U\rvert\leq\lvert K\rvert/2$, at least $\eta\lvert U\rvert/2$ vertices of $K-Q$ outside $U$ have a neighbor in $U$. As $1+x\geq e^{x/2}$ for $0\leq x\leq 1$, for each vertex $v$ of $K-Q$ the ball of radius $j$ around $v$ in $K-Q$ grows by a factor at least $1+\eta/2$ in each step while it has at most $\lvert K\rvert/2$ vertices. As $(1+\eta/2)^{r/2}\geq e^{\eta r/8}=e^{L}>\lvert K\rvert$, the balls of radius $r/2$ around two distinct vertices of $K-Q$ have more than $\lvert K\rvert/2$ vertices each. So they intersect, and this gives a path of length at most $r$.
\end{proof}

In the next lemma, a \emph{demand} inside a graph $J_i$ asks to join two distinct vertices of $J_i$ by a path in $J_i$, and a \emph{bundle} $B$ consists of a set $Y_B$ of vertices and some demands whose ends lie in $Y_B$.

\begin{lemma}\label{lem:routing}
Let $t\geq 1$, let $L$ be a positive integer, and put $r=2^8L^3$. Let $J_1,\dots,J_q$ be edge-disjoint graphs, possibly with common vertices, such that each $J_i$ is an $s$-expander on at most $2^L$ vertices, where $s\geq 6rt$. Let $\mathcal{B}$ be a finite family of bundles such that $\lvert Y_B\rvert\leq t$ and the demands of $B$ have pairwise disjoint ends for each $B\in\mathcal{B}$. Suppose that each vertex is an end of at most $t$ demands of all the bundles together. Then there is a family $\mathcal{B}_0\subseteq\mathcal{B}$ with
$$\lvert\mathcal{B}\setminus\mathcal{B}_0\rvert\leq\frac{rt^2}{s}\sum_i\lvert J_i\rvert$$
such that the demands of the bundles in $\mathcal{B}_0$ can be joined by paths with the following properties.
\begin{enumerate}[label=\upshape(\alph*)]
\item\label{rt:edge} The paths are pairwise edge-disjoint.
\item\label{rt:int} The paths of the demands of each bundle $B$ have no internal vertex in $Y_B$ and pairwise disjoint sets of internal vertices.
\end{enumerate}
\end{lemma}

\begin{proof}
Let $\mathcal{B}_0\subseteq\mathcal{B}$ be a maximal family such that the demands of its bundles can be joined by paths of length at most $r$ with properties~\ref{rt:edge} and~\ref{rt:int}, and fix such paths. For each $i$, let $F_i$ be the set of edges of these paths in $J_i$. At most $t\lvert J_i\rvert/2$ demands are inside $J_i$, so $\lvert F_i\rvert\leq rt\lvert J_i\rvert/2\leq s\lvert J_i\rvert/12$, as $s\geq 6rt$. Let $D_i$ be given by Lemma~\ref{lem:core} for $J_i$ and $F_i$.

Suppose that for some bundle $B\in\mathcal{B}\setminus\mathcal{B}_0$, no demand of $B$ inside any graph $J_i$ has an end in $D_i$. We join the demands of $B$ one by one. For a demand to join $x$ and $y$ inside $J_i$, let $Q$ consist of the vertices of $Y_B\setminus\{x,y\}$ and the internal vertices of the paths already chosen for $B$. Each demand of $B$ has two ends, and the ends of distinct demands of $B$ are distinct vertices of $Y_B$, so $B$ has at most $t/2$ demands. Each path already chosen has length at most $r$, and so fewer than $r$ internal vertices. Hence $\lvert Q\rvert\leq t+tr/2\leq rt\leq s/2$, as $r\geq 2$ and $s\geq 6rt$. The vertices $x$ and $y$ are not in $D_i$, and they are not in $Q$, since these paths have no internal vertex in $Y_B$. By Lemma~\ref{lem:core} applied to $(Q\cap V(J_i))\setminus D_i$, the graph $(J_i-F_i)-D_i-Q$ contains a path of length at most $r$ from $x$ to $y$. These paths have no edge in $F_i$, and the graphs $J_i$ are edge-disjoint, so they are edge-disjoint from the fixed paths. The paths of $B$ have disjoint sets of internal vertices which avoid $Y_B$, and their ends are distinct vertices of $Y_B$, so they are vertex-disjoint. Thus we can add $B$ to $\mathcal{B}_0$, contrary to the maximality of $\mathcal{B}_0$.

Hence each bundle in $\mathcal{B}\setminus\mathcal{B}_0$ has a demand inside some graph $J_i$ with an end in $D_i$. Each vertex is an end of at most $t$ demands, so
$$\lvert\mathcal{B}\setminus\mathcal{B}_0\rvert\leq t\sum_i\lvert D_i\rvert\leq\frac{2t}{s}\sum_i\lvert F_i\rvert\leq\frac{rt^2}{s}\sum_i\lvert J_i\rvert.$$
This proves the lemma.
\end{proof}

\subsection{Short and medium cycles}\label{sec:cycles}

Here we prove the three lemmas that Section~\ref{sec:qoh-proof} uses besides the results above. They close pieces of paths, pairs of edges and unpaired edges into cycles.

\begin{lemma}\label{lem:short-cycle}
A multigraph $\Gamma$ on $m\geq 1$ vertices with more than $2m$ edges contains a loop, two parallel edges or a cycle of length at most $2\log m+3$.
\end{lemma}

\begin{proof}
Suppose $\Gamma$ has no loop and no parallel edges. We repeatedly remove a vertex of degree at most two until no such vertex remains. As $e(\Gamma)>2m$ and we remove at most two edges with each vertex, a nonempty subgraph of minimum degree at least three remains. If it has no cycle of length at most $2\log m+3$, then the ball of radius $1+\log m$ around any of its vertices $v$ induces a tree. Indeed, otherwise an edge inside the ball and outside a breadth-first search tree from $v$ closes a cycle of length at most $2(1+\log m)+1=2\log m+3$. In this tree, the root has at least three children and every non-root vertex of depth at most $\log m$ has at least two children. So the tree has at least $3\cdot 2^{j-1}$ vertices at each depth $j\leq 1+\log m$. Some such $j$ exceeds $\log m$, so the tree has more than $3m/2>m$ vertices, a contradiction.
\end{proof}

\begin{lemma}\label{lem:grouping}
Let $d\geq 1$ be an integer and let $B$ be a bipartite graph with parts $W$ and $A$ such that $d_B(u)\leq d$ for each $u\in A$. Suppose that at each $w\in W$, the incident edges are paired into $\lfloor d_B(w)/2\rfloor$ pairs, so that each pair forms a cherry with center $w$. Then these cherries can be partitioned into at most $\lvert A\rvert+2d$ groups, each consisting of at most $d$ pairwise vertex-disjoint cherries.
\end{lemma}

\begin{proof}
Two cherries conflict if they share a vertex. A cherry $u\,w\,u'$ conflicts with at most $\lvert A\rvert/2-1$ other cherries with center $w$, as $w$ has at most $\lvert A\rvert$ edges in $B$, and with at most $2(d-1)$ cherries through $u$ or $u'$. So we can partition the cherries greedily into at most $\lvert A\rvert/2+2d$ classes such that no two cherries in a class conflict. We split each class with $a$ cherries into $\lceil a/d\rceil\leq a/d+1$ groups of size at most $d$. As there are at most $e(B)/2\leq d\lvert A\rvert/2$ cherries, the number of groups is at most $\lvert A\rvert/2+(\lvert A\rvert/2+2d)=\lvert A\rvert+2d$.
\end{proof}

\begin{lemma}\label{lem:medium-cycle}
Let $M$ be a graph on $m\geq 2$ vertices and let $\lambda\geq 1$ be an integer. Then $E(M)$ can be partitioned into cycles of lengths between $\lambda$ and $\lambda+2^{12}\log^3 m$ and a graph with at most $8(\lambda+1)m\log m$ edges.
\end{lemma}

\begin{figure}[ht]
\centering
\begin{minipage}[c]{0.36\textwidth}
\centering
\begin{tikzpicture}[scale=0.75]
\foreach \i/\x in {0/0.5,1/1.5,2/2.5,3/3.5,4/4.5,5/5.5}{\node[vtx] (v\i) at (\x,2.4) {};}
\draw[red!70!black,thick] (v0)--(v1)--(v2); \draw[red!70!black,thick,dotted] (v2)--(v3); \draw[red!70!black,thick] (v3)--(v4)--(v5);
\node[font=\footnotesize,above] at (v0.north) {$v_0$};
\node[font=\footnotesize,above] at (v1.north) {$v_1$};
\node[font=\footnotesize,above] at (v5.north) {$v_\lambda$};
\draw[decorate,decoration={brace,amplitude=4pt}] (0.3,2.95)--(5.7,2.95) node[midway,above=3pt,font=\small]{$X$};
\node[vtx] (s0) at (1.1,0.6) {}; \node[vtx] (s0b) at (0.6,0.3) {};
\node[vtx] (s1) at (4.9,0.6) {}; \node[vtx] (s1b) at (5.4,0.3) {};
\draw[red!70!black,thick] (v0)--(s0); \draw[gray] (v0)--(s0b);
\draw[red!70!black,thick] (v5)--(s1); \draw[gray] (v5)--(s1b);
\node[font=\footnotesize] at (0.85,-0.0) {$S_0$};
\node[font=\footnotesize] at (5.15,-0.0) {$S_1$};
\draw[blue!70!black,thick,dashed] (s0) .. controls (1.9,0.9) and (2.4,0.1) .. (3.0,0.1) .. controls (3.6,0.1) and (4.1,0.9) .. (s1);
\node[font=\footnotesize] at (3.0,-0.5) {$J-X$};
\end{tikzpicture}
\end{minipage}\hfill
\begin{minipage}[c]{0.6\textwidth}
\setlength{\abovecaptionskip}{0pt}
\caption{The cycle in the proof of Lemma~\ref{lem:medium-cycle}. The dashed path in $J-X$ closes the path $v_0\cdots v_\lambda$ into a cycle.}
\label{fig:medium}
\end{minipage}
\end{figure}
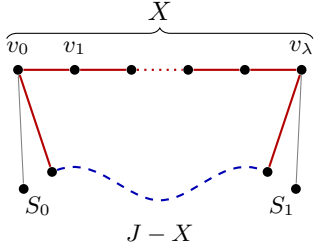

\begin{proof}
Take a maximal family of edge-disjoint cycles with lengths in this range and let $M'$ be the graph of the remaining edges. Suppose to the contrary that $e(M')>8(\lambda+1)m\log m$. We apply Theorem~\ref{thm:bm}\ref{bm:split} to $M'$ with $2\lambda+2$ playing the role of $s$. As $e(M')>8(\lambda+1)m\log m$, some edge of $M'$ is not removed, and the part which contains it is a $(2\lambda+2)$-expander $J\subseteq M'$. As $\delta(J)>2\lambda+2$, we can greedily choose a path $v_0v_1\cdots v_{\lambda}$ in $J$. Let $X=\{v_0,\dots,v_{\lambda}\}$, and let $S_0$ and $S_1$ be the sets of neighbors of $v_0$ and of $v_{\lambda}$ outside $X$. Both are nonempty. Figure~\ref{fig:medium} shows the cycle we find.

Let $u_0\in S_0$ and $u_1\in S_1$. If $u_0=u_1$, then $v_0\cdots v_\lambda u_0$ is a cycle of length $\lambda+2$. Otherwise, we apply Lemma~\ref{lem:core} to $J$ with $L=\lceil\log m\rceil$ and $F=\varnothing$, as $\lvert J\rvert\leq m\leq 2^L$. Then $D=\varnothing$, and as $\lvert X\rvert=\lambda+1$ is half of $2\lambda+2$, the graph $J-X$ contains a path of length at most $2^8L^3\leq 2^{11}\log^3 m$ from $u_0$ to $u_1$. Together with the path $v_0\cdots v_{\lambda}$ and the edges $v_0u_0$ and $v_\lambda u_1$, it forms a cycle of length at least $\lambda+2$ and at most $\lambda+2^{12}\log^3 m$. This contradicts the maximality of the family.
\end{proof}

\subsection{Proof of \texorpdfstring{Theorem~\ref{thm:qoh}}{Theorem 1.4}}\label{sec:qoh-proof}

We now prove Theorem~\ref{thm:qoh}. Let $\eps>0$. We may assume $\eps\leq 1$, and we put $\alpha=\eps/3$. Let $n_0$ be sufficiently large, and let $G$ be a graph on $n$ vertices with $\delta(G)\geq n^{\eps}$. If $n<n_0$, then the fewer than $n_0n/2$ edges of $G$, recorded as single edges, form a decomposition. So we may assume that $n\geq n_0$. We put
$$\ell=\log n,\qquad s=n^{2\alpha}\qquad\text{and}\qquad \lambda=\ell^{300}.$$

\medskip\noindent\textbf{Expanders and the graph $H$.} By Theorem~\ref{thm:bm}\ref{bm:split} applied to $G$, there is a set $F$ of at most $4sn\ell$ edges whose removal leaves edge-disjoint $s$-expanders $J_1,\dots,J_{j^*}$ with $\sum_i m_i\leq 2n$, where $m_i=\lvert J_i\rvert$. Each $J_i$ has minimum degree greater than $s$, so $m_i>s$ and $j^*\leq 2n/s$. Let $Z$ be the set of vertices which lie in no $J_i$. All edges at a vertex of $Z$ belong to $F$, so, as $n$ is large,
$$\lvert Z\rvert\leq\frac{2\lvert F\rvert}{\delta(G)}\leq\frac{8sn\ell}{n^{3\alpha}}=\frac{8n\ell}{n^{\alpha}}\leq\frac{n}{\lambda^3}.$$

\smallskip\noindent\emph{Reservoirs.} Recall that a demand inside $J_i$ asks to join two distinct vertices of $J_i$ by a path in $J_i$. We call such a path a \emph{joining path}. The joining paths inside $J_i$ will use only edges of a sparse random subgraph $R_i$ of $J_i$, its \emph{reservoir}. The edges of $R_i$ which are not used are decomposed at the end by Theorem~\ref{thm:cfs}, so $R_i$ must be sparse. Put $h=n^{\alpha/2}$.

\begin{claim}\label{cl:reservoirs}
For each $i$, there is a spanning subgraph $R_i$ of $J_i$ which is an $n^{\alpha}$-expander and is $(\alpha/4,8/\alpha)$-sparse.
\end{claim}

\begin{proof}[Proof of Claim~\ref{cl:reservoirs}]
Let $R_i$ contain each edge of $J_i$ independently with probability $1/h$. We apply Lemma~\ref{lem:split} to the $s$-expander $J_i$ with $h$. Since $J_i\subseteq G$, we have $m_i\leq n$ and so $\log m_i\leq\ell$. Thus its condition $s\geq 64h\log m_i$ holds, and $s/(2h)\geq n^{\alpha}$, as $n$ is large. So $R_i$ is an $n^{\alpha}$-expander, except with probability at most $m_i^{-5}$. The graph $R_i$ has the same distribution as $J_i\cap G(m_i,1/h)$, where $G(m_i,1/h)$ is the random graph on $V(J_i)$. As $1/h\leq m_i^{-\alpha/4}$, Theorem~\ref{thm:cfs}\ref{cfs:random} with $\rho=\alpha/4$ shows that this random graph is $(\alpha/4,8/\alpha)$-sparse with probability $1-o(1)$, and then so is its subgraph $R_i$. As $m_i\geq s$ is large, some choice has both properties. This proves the claim.
\end{proof}

We fix such reservoirs, and we write $R$ for the graph on $V(G)$ whose edges are the edges of all reservoirs. The reservoirs are edge-disjoint, as the graphs $J_i$ are.

\smallskip\noindent\emph{Removing cycles.} By Theorem~\ref{thm:bm}\ref{bm:bulk} applied to $G-E(R)$, we can remove at most $6n$ edge-disjoint cycles from $G-E(R)$ so that the remaining graph $H$ satisfies $e(H)\leq\beta n\ell^{274}\leq n\lambda$. Thus $H$ has no edge in $R$.

A vertex $u$ will be an end of up to $d_H(u)$ paths inside the expanders, but Lemma~\ref{lem:routing} needs a bound on the number of paths with a given end. So we treat the vertices of large degree in $H$ separately, together with the vertices which lie in no expander. Let
$$W=Z\cup\{v:d_H(v)>\lambda^3\}\qquad\text{and}\qquad X_1=V(G)\setminus W.$$
We have $\lvert W\rvert\leq n/\lambda^3+2e(H)/\lambda^3\leq 3n/\lambda^2$. Each vertex of $X_1$ has degree at most $\lambda^3$ in $H$. We assign each vertex $v\in X_1$ to one expander $J_{i(v)}$ which contains $v$, and we write $A_i$ for the set of vertices assigned to $J_i$. The sets $A_i$ are disjoint.

Let $H_1=H[W]$ and $H_2=H[X_1]$, and let $H_3$ be the graph of the edges of $H$ between $W$ and $X_1$. If $\lvert W\rvert\geq 2$, then Theorem~\ref{thm:bm}\ref{bm:bulk} applied to $H_1$ shows that $H_1$ can be decomposed into at most $(3n/\lambda^2)(6+\lambda)\leq n$ cycles and edges, as $\lvert W\rvert\leq 3n/\lambda^2$ and $\beta\ell^{274}\leq\lambda$.

\medskip\noindent\textbf{Constructions.} We now close most edges of $H_2$ and $H_3$ into cycles. A \emph{construction} $C$ is a cyclic sequence of paths $P_1,\dots,P_k$ of $H$ together with indices $i_1,\dots,i_k\in[j^*]$ such that $P_p$ runs from a vertex $x_p\in A_{i_p}$ to a vertex $y_p\in A_{i_{p+1}}$, where indices are read modulo $k$. The paths are vertex-disjoint, except that $y_{p-1}=x_p$ is allowed. For each $p$ with $y_{p-1}\neq x_p$, there is a demand to join $y_{p-1}$ to $x_p$ inside $J_{i_p}$. We write $Y_C$ for the set of vertices of the paths $P_1,\dots,P_k$ of $C$. If each demand of $C$ is joined by a joining path, and these joining paths have pairwise disjoint interiors which avoid $Y_C$ and use no edge of $H$, then they form a cycle together with $P_1,\dots,P_k$. Each construction $C$ will satisfy $\lvert Y_C\rvert\leq 4\lambda^3$. We find constructions of three kinds. Those of the first kind come from pieces of paths of $H_2$, those of the second kind from pairs of edges of $H_3$ at vertices of $W$, and those of the third kind from the remaining edges of $H_3$ (see Figure~\ref{fig:constructions}(a)--(c), respectively).

\begin{figure}[ht]
\centering
\begin{tikzpicture}[scale=0.9]
\begin{scope}[scale=0.7]
\draw[blue!50!black,fill=blue!5] (0,0) ellipse (1.45 and 0.95);
\fill[blue!5] (4.6,0.8) ellipse (1.45 and 0.95);
\fill[blue!5] (4.6,-0.8) ellipse (1.45 and 0.95);
\draw[blue!50!black] (4.6,0.8) ellipse (1.45 and 0.95);
\draw[blue!50!black] (4.6,-0.8) ellipse (1.45 and 0.95);
\node[font=\footnotesize,blue!50!black] at (-1.25,0.95) {$J_{i_1}$};
\node[font=\footnotesize,blue!50!black] at (5.85,1.75) {$J_{i_2}$};
\node[font=\footnotesize,blue!50!black] at (5.85,-1.75) {$J_{i_3}$};
\coordinate (a1) at (0.95,0.45); \coordinate (a2) at (0.95,-0.45);
\coordinate (b1) at (3.55,0.95); \coordinate (b2) at (5.65,0.95);
\coordinate (c1) at (5.65,-0.95); \coordinate (c2) at (3.55,-0.95);
\coordinate (r1) at (-0.15,-0.5); \coordinate (r2) at (-0.45,0.3);
\coordinate (g1) at (4.2,1.35); \coordinate (g2) at (5.05,1.3);
\coordinate (o1) at (5.05,-1.3); \coordinate (o2) at (4.2,-1.35);
\draw[blue!70!black,thick,dashed] (a2)--(r1)--(r2)--(a1);
\draw[blue!70!black,thick,dashed] (b1)--(g1)--(g2)--(b2);
\draw[blue!70!black,thick,dashed] (c1)--(o1)--(o2)--(c2);
\foreach \q in {r1,r2,g1,g2,o1,o2}{\node[circle,fill=blue!70!black,inner sep=1.4pt] at (\q) {};}
\coordinate (p1) at (1.9,1.15); \coordinate (p2) at (2.7,1.25);
\coordinate (p3) at (6.55,0.55); \coordinate (p4) at (6.55,-0.55);
\coordinate (p5) at (2.7,-1.25); \coordinate (p6) at (1.9,-1.15);
\draw[red!70!black,thick] (a1)--(p1)--(p2)--(b1);
\draw[red!70!black,thick] (b2)--(p3)--(p4)--(c1);
\draw[red!70!black,thick] (c2)--(p5)--(p6)--(a2);
\foreach \q in {a1,a2,b1,b2,c1,c2,p1,p2,p3,p4,p5,p6}{\node[circle,fill=blue!60!black,inner sep=1.4pt] at (\q) {};}
\node[font=\footnotesize] at (2.3,1.55) {$P_1$};
\node[font=\footnotesize] at (6.9,0) {$P_2$};
\node[font=\footnotesize] at (2.3,-1.55) {$P_3$};
\end{scope}
\node[font=\small] at (1.9,-2.2) {(a) pieces of paths};
\begin{scope}[xshift=6.9cm]
\draw[blue!50!black,fill=blue!5] (1,0) ellipse (2.3 and 0.8);
\node[font=\footnotesize,blue!50!black] at (1,-1.05) {$J_i$};
\node[upvtx,label={[font=\footnotesize]above:$w_1$}] (w1) at (-0.1,1.6) {};
\node[upvtx,label={[font=\footnotesize]above:$w_2$}] (w2) at (2.1,1.6) {};
\node[vtx,label={[font=\scriptsize]left:$u_1$}] (x1) at (-0.5,0.05) {};
\node[vtx,label={[font=\scriptsize]above right:$u'_1$}] (y1) at (0.3,0.05) {};
\node[vtx,label={[font=\scriptsize]above left:$u_2$}] (x2) at (1.7,0.05) {};
\node[vtx,label={[font=\scriptsize]right:$u'_2$}] (y2) at (2.5,0.05) {};
\draw[red!70!black,thick] (x1)--(w1)--(y1); \draw[red!70!black,thick] (x2)--(w2)--(y2);
\draw[blue!70!black,thick,dashed] (y1) to[bend right=40] (x2);
\draw[blue!70!black,thick,dashed] (y2) to[bend left=50] (x1);
\node[font=\small] at (1,-2.2) {(b) pairs of edges at $W$};
\end{scope}
\begin{scope}[xshift=12.1cm]
\draw[blue!50!black,fill=blue!5] (0,0) ellipse (0.8 and 0.55);
\draw[blue!50!black,fill=blue!5] (2.4,0) ellipse (0.8 and 0.55);
\node[font=\footnotesize,blue!50!black] at (0,-0.85) {$J_{i_1}$};
\node[font=\footnotesize,blue!50!black] at (2.4,-0.85) {$J_{i_2}$};
\node[upvtx,label={[font=\footnotesize]above:$w_1$}] (u1) at (1.2,1.5) {};
\node[upvtx,label={[font=\footnotesize]below:$w_2$}] (u2) at (1.2,-1.5) {};
\node[vtx] (p1) at (-0.3,0.1) {}; \node[vtx] (q1) at (0.3,-0.1) {};
\node[vtx] (p2) at (2.1,-0.1) {}; \node[vtx] (q2) at (2.7,0.1) {};
\draw[red!70!black,thick] (u1)--(p1); \draw[red!70!black,thick] (q1)--(u2);
\draw[red!70!black,thick] (u2)--(p2); \draw[red!70!black,thick] (q2)--(u1);
\draw[blue!70!black,thick,dashed] (p1) to[bend right=30] (q1);
\draw[blue!70!black,thick,dashed] (p2) to[bend left=30] (q2);
\node[font=\small] at (1.2,-2.2) {(c) unpaired edges};
\end{scope}
\end{tikzpicture}
\caption{The three kinds of constructions, with the paths of $H$ solid and the joining paths inside the expanders dashed. In (a), $J_{i_2}$ and $J_{i_3}$ share vertices. In (c), drawn schematically, the construction comes from a cycle $i_1w_1i_2w_2i_1$ of the auxiliary bipartite graph $B_{\mathrm{odd}}$.}
\label{fig:constructions}
\end{figure}
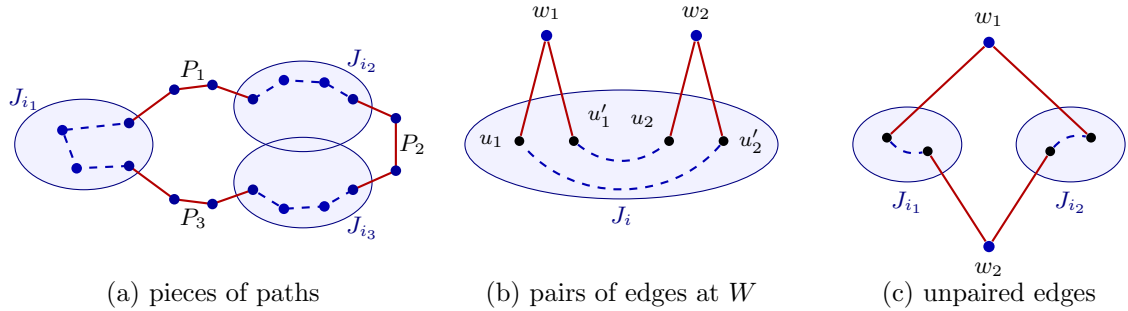

\smallskip\noindent\emph{Pieces of paths.} By Theorem~\ref{thm:bm}\ref{bm:paths}, we decompose $H_2$ into at most $n$ paths. All vertices of these paths lie in $X_1$.

We cut each path into pieces of length at most $\lambda+1$, so that each piece except the last one of its path has length at least $\lambda$. So the number of pieces is at most $n+e(H)/\lambda\leq 2n$. Each vertex $v\in X_1$ lies in at most $d_H(v)\leq\lambda^3$ pieces.

We greedily partition the pieces into at most $(\lambda+2)\lambda^3$ classes of pairwise vertex-disjoint pieces, so that the pieces of one class can form a construction. This is possible since a piece has at most $\lambda+2$ vertices and each of them lies in at most $\lambda^3$ pieces. For each class, consider the multigraph $\Gamma$ on the set $[j^*]$ of indices of the expanders $J_1,\dots,J_{j^*}$, with an edge between $i(x)$ and $i(y)$ for each piece of this class with ends $x$ and $y$. Recall that $x$ is assigned to $J_{i(x)}$. While $\Gamma$ has more than $2j^*$ edges, Lemma~\ref{lem:short-cycle} gives a loop, two parallel edges or a cycle of length at most $2\log j^*+3\leq 2\ell+3$. Let $i_1,\dots,i_k$ be the vertices of this loop, these parallel edges or this cycle in cyclic order, and let $P_1,\dots,P_k$ be the corresponding pieces, so that $P_p$ runs from a vertex of $A_{i_p}$ to a vertex of $A_{i_{p+1}}$. These pieces form a construction $C$, and we remove the corresponding edges from $\Gamma$. As each piece has at most $\lambda+2$ vertices, we have $\lvert Y_C\rvert\leq k(\lambda+2)\leq(2\ell+3)(\lambda+2)\leq 4\lambda^3$.

When at most $2j^*$ edges of $\Gamma$ remain, we record the edges of the remaining pieces of this class as single edges. There are at most $(\lambda+2)\lambda^3$ classes and $j^*\leq 2n/s$, so in total this records at most $2j^*(\lambda+2)\lambda^3(\lambda+1)\leq 4n(\lambda+2)^2\lambda^3/s\leq n$ edges, as $s=n^{2\alpha}$. There are at most $2n$ constructions of the first kind, since each of them uses at least one piece.

\smallskip\noindent\emph{Pairs of edges at $W$.} Fix $i$, and let $B_i$ be the graph of the edges of $H_3$ between $W$ and $A_i$. Each vertex of $A_i$ has degree at most $\lambda^3$ in $B_i$. At each $w\in W$, we pair its edges in $B_i$ arbitrarily, and each pair forms a cherry with center $w$. If $d_{B_i}(w)$ is odd, then one edge $w\,v(w,i)$ remains unpaired. The number of these cherries may be much larger than $n$, so we close many of them into one cycle. Lemma~\ref{lem:grouping} applied to $B_i$ with $A_i$ and $\lambda^3$ playing the roles of $A$ and $d$ partitions these cherries into at most $\lvert A_i\rvert+2\lambda^3$ groups. Let $u_p\,w_p\,u'_p$, $p\in[k]$, be the cherries of a group. They are pairwise vertex-disjoint, so together with $i_1=\dots=i_k=i$ they form a construction $C$ with $\lvert Y_C\rvert=3k\leq 3\lambda^3$. The number of constructions of the second kind is at most $\sum_i(\lvert A_i\rvert+2\lambda^3)\leq n+2\lambda^3j^*\leq 2n$, as the sets $A_i$ are disjoint and $j^*\leq 2n/s$.

\smallskip\noindent\emph{Unpaired edges.} Consider the bipartite graph $B_{\mathrm{odd}}$ with parts $W$ and $[j^*]$ in which $w$ and $i$ are adjacent exactly when $d_{B_i}(w)$ is odd. We associate the edge $wi$ with the edge $w\,v(w,i)$ of $H$. The graph $B_{\mathrm{odd}}$ is simple and has $n'=\lvert W\rvert+j^*\leq 4n/\lambda^2$ vertices, as $\lvert W\rvert\leq 3n/\lambda^2$ and $j^*\leq 2n/s$. We may assume that $B_{\mathrm{odd}}$ has an edge, so $n'\geq 2$. By Lemma~\ref{lem:medium-cycle} applied to $B_{\mathrm{odd}}$ with $\lambda$, the edges of $B_{\mathrm{odd}}$ can be partitioned into at most $e(B_{\mathrm{odd}})/\lambda\leq e(H)/\lambda\leq n$ cycles of lengths between $\lambda$ and $\lambda+2^{12}\ell^3$ and at most $8(\lambda+1)n'\log n'\leq 64n\ell/\lambda\leq n$ remaining edges. We record the edges of $H$ associated with the remaining edges as single edges, and we turn each of the cycles into a construction as follows.

Write such a cycle as $i_1w_1i_2w_2\cdots i_kw_ki_1$, where $k\geq 2$ and the vertices $w_p$ and the indices $i_p$ are distinct. The paths $x_p\,w_p\,y_p$ for $p\in[k]$, where $x_p=v(w_p,i_p)$ and $y_p=v(w_p,i_{p+1})$, form a construction $C$. Indeed, the vertices $w_p$ are distinct and lie in $W$, and the vertices $x_p$ and $y_p$ lie in $X_1$. As the indices $i_p$ are distinct, two of the vertices $x_p$ and $y_p$ lie in a common set $A_i$ only if they are $y_{p-1}$ and $x_p$ for some $p$. As the cycle of $B_{\mathrm{odd}}$ has length $2k$, we have $\lvert Y_C\rvert\leq 3k\leq\tfrac32(\lambda+2^{12}\ell^3)\leq 4\lambda^3$.

Distinct constructions use disjoint sets of edges of $H$. Within one construction, the demands have pairwise disjoint ends. Indeed, a vertex lies on two of the paths $P_1,\dots,P_k$ only if it is $y_{p-1}=x_p$ for some $p$, and there is no demand to join $y_{p-1}$ to $x_p$ in that case. We assign to each end $u$ of a demand the edge at $u$ of the path of the construction which ends at $u$. So distinct demands with a common end $u$ get distinct edges at $u$. Hence each vertex $u\in X_1$ is an end of at most $d_H(u)\leq\lambda^3$ demands of the constructions, and all of them are to be joined inside $J_{i(u)}$.

\medskip\noindent\textbf{Joining paths.} We apply Lemma~\ref{lem:routing} to the reservoirs $R_1,\dots,R_{j^*}$ with $n^{\alpha}$ playing the role of $s$, $L=\ell$ and $t=4\lambda^3$. For each construction $C$, we take the bundle $B$ with $Y_B=Y_C$ whose demands are the demands of $C$, where the demand to join $y_{p-1}$ to $x_p$ inside $J_{i_p}$ is regarded as a demand inside $R_{i_p}$. The reservoirs are edge-disjoint $n^{\alpha}$-expanders on at most $2^L$ vertices. As shown above, the demands of each construction $C$ have disjoint ends in $Y_C$, each vertex is an end of at most $\lambda^3$ demands, and $\lvert Y_C\rvert\leq 4\lambda^3$. As $r=2^8L^3$ and $\lambda$ are polylogarithmic in $n$, we have $n^{\alpha}\geq 24r\lambda^3=6rt$ for large $n$. So Lemma~\ref{lem:routing} gives joining paths for the demands of all constructions except at most $(16r\lambda^6/n^{\alpha})\sum_i m_i\leq 32r\lambda^6n^{1-\alpha}$ of them, which we say \emph{fail}. A construction $C$ has at most $\lvert Y_C\rvert$ edges, as the union of its paths is a cycle or a disjoint union of paths with vertex set $Y_C$. We record the edges of the failed constructions as single edges, at most $128r\lambda^9n^{1-\alpha}\leq n$ edges in total.

\medskip\noindent\textbf{The decomposition.} The joining paths use edges of the reservoirs, so they use no edge of $H$. By Lemma~\ref{lem:routing}\ref{rt:int}, the joining paths of a construction $C$ which does not fail have disjoint interiors which avoid $Y_C$, so $C$ gives a cycle. These cycles are edge-disjoint, since distinct constructions use disjoint sets of edges of $H$ and the joining paths are edge-disjoint.

Finally, for each $i$, let $R_i^*$ be the graph of the edges of $R_i$ which lie on no joining path. By Claim~\ref{cl:reservoirs} and Theorem~\ref{thm:cfs}\ref{cfs:dec}, $R_i^*$ can be decomposed into at most $b_\alpha m_i$ cycles and edges, where $b_\alpha=b(\alpha/4,8/\alpha)$.

Each edge of $G$ now lies in exactly one recorded cycle or edge. Indeed, each edge of $G-E(R)$ lies in a removed cycle or in $H$, and each edge of $R$ lies on a joining path or in a graph $R_i^*$. The edges of $H_1$ are decomposed by Theorem~\ref{thm:bm}\ref{bm:bulk}. Each edge of $H_2$ lies in a construction of the first kind or in a recorded piece. Each edge of $H_3$ lies in a construction of the second or third kind or is recorded as a single edge. The edges of a construction lie in the cycle given by it, or are recorded as single edges if it fails.

\smallskip\noindent\emph{Counting the parts.} There are at most $6n$ removed cycles and at most $n$ parts for $H_1$. We recorded at most $n$ edges of unused pieces, at most $n$ unpaired edges and at most $n$ edges of failed constructions. There are at most $2n$, $2n$ and $n$ constructions of the three kinds, and the graphs $R_i^*$ give at most $b_\alpha\sum_im_i\leq 2b_\alpha n$ parts. In total, $G$ is decomposed into at most $(15+2b_\alpha)n$ cycles and edges. This proves Theorem~\ref{thm:qoh} with $C_\eps=\max\{15+2b_{\eps/3},\,n_0/2\}$.

\section{Concluding remarks}\label{sec:conclusion}

We proved that every $n$-vertex graph can be decomposed into $O(n)$ cycles and edges. Our constant is large, and we did not try to optimize it. It would be interesting to determine the optimal constant, which is at least $3/2$ by the example of Erd\H{o}s~\cite{erdos83}. Haj\'os' conjecture and Gallai's conjecture on path decompositions remain open.

The proof of Theorem~\ref{thm:eg} reduces the problem to light graphs. This reduction uses only a bounded nondecreasing function $f$ with $f(d)-f(d-2)\geq d^{-3/2}$ for $d\geq 2$, and it might be useful for other decomposition and covering problems.

\phantomsection\addcontentsline{toc}{section}{Note added}
\section*{Note added}
After the first version of this paper was submitted to arXiv, and before it appeared, OpenAI announced an independent proof of Conjecture~\ref{conj:eg}~\cite{oai}. The two proofs follow different strategies. Their proof is an induction on the number of vertices. It splits the graph into pieces of large cut expansion of successively smaller sizes, identifies pairs of vertices in the pieces of one layer, applies induction to the resulting quotient graphs and lifts their cycles back to the original graph. Our proof consists of two separate statements, which may be of independent interest, and a short procedure which combines them. Theorem~\ref{thm:structure}, the main new ingredient, finds near-closed subgraphs of polynomial minimum degree in light graphs. Its proof is elementary and uses no expanders. Theorem~\ref{thm:qoh} extends the result of Conlon, Fox and Sudakov~\cite{cfs} from linear to polynomial minimum degree.

This version simplifies the first version. Heavy cycles are defined by degrees instead of vertex weights. This removes logarithmic factors from Sections~\ref{sec:prelim}--\ref{sec:structure}. In Section~\ref{sec:qoh}, a deterministic routing lemma (Lemma~\ref{lem:routing}) replaces the random coloring with $\Theta(\log^3 n)$ colors, and the edges of the expanders are no longer treated separately.

\phantomsection\addcontentsline{toc}{section}{Use of AI tools}
\section*{Use of AI tools}
The author made extensive use of ChatGPT (OpenAI) and Claude (Anthropic) in developing the arguments and preparing the text and figures of this paper. The author has verified all results and takes full responsibility for the content of the paper.

\phantomsection

\end{document}